\documentclass[12pt]{amsart}
\usepackage{amsmath,amssymb,hyperref,url}
\usepackage{cleveref,bm}

\usepackage{amsthm}
\newtheorem{thm}{Theorem}
\newtheorem{lemma}{Lemma}[section]
\newtheorem{prop}[lemma]{Proposition}
\newtheorem{cor}[lemma]{Corollary}
\theoremstyle{definition}

\newtheorem{exm}[lemma]{Example}

\usepackage{tikz}
\usetikzlibrary{patterns,arrows,snakes}

\usepackage{graphicx}

\usepackage{floatrow}

\def\mes{{\mathrm{mes}}}

\thanks{The author thanks Vered Rom-Kedar for posing this problem and for many useful discussions on this subject. The author is also much obliged to Dmitry Dolgopyat for enlightening comments and advice. The research was partially supported by NSF grant DMS 1665046 and by BSF grant 2016105. Part of this work was done at Weizmann Institute, where the hospitality and excellent working conditions are gratefully acknowledged.}

\begin{document}
\title{A Rectangular Billiard with Moving Slits}
\author{Jing Zhou}
\date{\today}
\maketitle

\begin{abstract}
    We describe an exponential Fermi accelerator in a two-dimensional billiard with a moving slit. 
    We have found a mechanism of trapping regions which provides the exponential acceleration for almost all initial conditions with sufficiently high initial energy. Under an additional hyperbolicity assumption, we estimate the waiting time after which most high-energy orbits start to gain energy exponentially fast.
\end{abstract}

\section{Introduction}
In an attempt to explain the existence of high energy particles in cosmic rays, Fermi \cite{Fermi} in 1949 proposed a model in which charged particles undergo repeated reflections in moving magnetic fields. 
Later in 1961 Ulam \cite{Ulam} proposed that a similar mechanism should appear in finite degree of freedom
systems. 
He described a toy model where a particle bounces elastically between two walls, one fixed and the other moving periodically. Ulam \cite{Ulam} performed numerical experiments on a piecewise linear model and conjectured that there exist escaping orbits whose energy tend to infinity with time. Since then extensive efforts have been made by both mathematicians and physicists to locate escaping orbits in various settings
(see \cite{LiLi, Dol-FA, GRKTChaos} for surveys on this subject).\\

Notably the KAM theory has eliminated the possibility of such escaping orbits in one-dimensional Fermi-Ulam model for sufficiently smooth wall motions \cite{LaLe} \cite{Pu83} \cite{Pu94}, as the prevalence of invariant curves forces all orbits to be bounded. However, unbounded solution can still be obtained in nonsmooth case. For example, Zharnitsky \cite{Zhar} has found a type of unbounded orbits in the Ulam's piecewise linear case, which grows linearly with time. In a one-dimensional Fermi-Ulam model with one discontinuity, De Simoi and Dolgopyat \cite{deSD} has showed that there exists a parameter that completely shapes the 
large energy behavior of the system: in the hyperbolic regime the escaping set has zero measure but full Hausdorff dimension while in some elliptic cases the escaping orbits have infinite measure.
There are many interesting question pertaining to 
Fermi-Ulam model with non-periodic wall motion. We refer the reader to \cite{Zha00, KO18}
for the results in quasi-periodic case.
One and a half degree of freedom models where the motion between collision is not free but
is subjected to a potential are discussed, for example in 
\cite{Pu77, Or99, Or02, Dol08, DS09, AZ}.
\\

Two-dimensional moving billiard models are natural generalization of the one-dimensional Fermi acceleration model and can often provide chaotic orbits even in the smooth case. For example, unbounded orbits were found in billiard models with the smoothly breathing boundary \cite{GT, KMKC, LKSK}. In a Lorentz gas model \cite{LRA}, the average velocity of particles grows linearly in time for stochastic perturbation of scatterer boundaries and quadratically for periodic perturbation. 
Exponential growing orbits for non-autonomous billiards were constructed  in \cite{GT2} 
but in general it remains challenging to detect a positive measure set of exponentially growing orbits.   
Exponential acceleration is also conjectured to be generic for oscillating mushrooms
\cite{GRKT14}.
The models which are closest to our setting are the following.
 Shah, Turaev and Rom-Kedar \cite{STR} investigated a rectangular billiard model with a moving slit. They proposed a random process approximation for the particle energy from which the calculated exponential growth rate agrees well with the numerical results for typical trajectories in the non-resonant case. They also observed numerically in \cite{STR} that the growth rate in the resonant case is significantly higher than that in the non-resonant case. Later they generalized the idea to higher-dimensional time-dependent billiards and obtained a new class of numerically robust exponential accelerators \cite{GRST}. The model presented in this paper was inspired by their work.\\

We study a rectangular billiard with two moving slits.
In our model, the billiard table is a unit square. Two slits are moving vertically in the table with the length of left slit $\lambda$ and the length of the right slit $1-\lambda$, where $\lambda \in (0,1)$ is a constant. The motion of the two slits are described by two $C^2$ 2-periodic functions $f_L(t)$ and $f_R(t)$ respectively. A massless ball bounces elastically against the moving slits as well as the boundary of the rectangular table.\\
\begin{figure}[h!]
   \centering
      \begin{tikzpicture}
         \draw (-2,-2) rectangle (2,2);
         \draw[pattern=north west lines] (-2,0) rectangle (-0.5,0.2)
                                         (-0.5,-0.5) rectangle (2,-0.3);
         \draw (-2,-2) node[anchor=north east]{0}
               (-2,2) node[anchor=east]{1}
               (-0.5,-2) node[anchor=north]{$\lambda$}
               (2,-2) node[anchor=north]{1}
               (-2.3,0.1) node[anchor=east]{$f_L(t)$}
               (2.3,-0.4) node[anchor=west]{$f_R(t)$};
         \draw[dotted] (-0.5,0) -- (-0.5,-2);
         \draw[<->] (-2.2,-0.2) -- (-2.2,0.4);
         \draw[<->] (2.2,-0.7) -- (2.2,-0.1);
         \draw (0.2,1) circle (0.1cm);
      \end{tikzpicture}
   \caption{Rectangular Billiard with Moving Slits}
\end{figure}
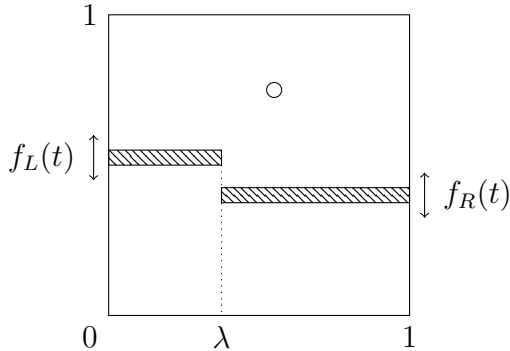

In this paper we study the simplest resonant case. Namely,
we assume that the horizontal speed the ball is 1, so the horizontal coordinate of the ball is periodic with period 2.
Hence we have 1:1 resonance between the period of moving slits and the 
period the horizontal motion of the ball.  
The ball experiences two jumps between the left and the right parts of the table during each period. We denote by $x_0$ the starting horizontal position of the ball. We assume without loss of generality that the ball starts from the left part, i.e. $0\le x_0 < \lambda$. The ball jumps from the left slit to right one at time $t_1^*=\lambda-x_0$ and then from right to the left at time $t_2^*=2-\lambda-x_0$.

We record the time and the vertical velocity of the ball immediately after each collision with the slits. We exclude from our discussion the trajectories having a collision at $x=\lambda.$ The excluded orbits constitute a measure zero set among all the initial conditions.\\

We describe in this paper a new exponential accelerator. We show that
 almost all initial conditions with sufficiently high initial velocity produce exponential energy growth in the future, provided that the relative positions of the two slits change at the time of the two jumps between left and right parts of the table. Moreover, under an additional hyperbolicity assumption we estimate the waiting time after which most high-energy orbits start to accelerate exponentially.

\section{Main Results}
In this section we describe the exponential accelerator we have found.\\
For a wide range of choices in $\lambda$ and $x_0$, the relative positions of the left and right slits change when the ball jumps from one slit to the other at time $t_1^*$ and $t_2^*$ (c.f. Figure \ref{fig2}). A trapping region is created in this case and the ball starts to gain energy exponentially fast once it gets trapped.

\begin{thm}\label{thm1}
Assume that $f_L(t_1^*)<f_R(t_1^*)$ and $f_L(t_2^*)>f_R(t_2^*)$
or $f_L(t_1^*)>f_R(t_1^*)$ and $f_L(t_2^*)<f_R(t_2^*)$.
Then there exists $V_* \gg 1$, which depends only on $f_L$ and $f_R$, such that almost every orbit 
whose initial speed is greater than $V^*$ eventually gains energy exponentially in time. 
In particular, the set of initial conditions $(t_0,v_0)$ which do not enjoy exponential energy growth 
has finite measure.
\end{thm}

In the presence of a trapping region, with additional hyperbolicity assumptions we estimate 
the waiting time after which most high-energy orbits start to accelerate exponentially.\\
We define a new function $f$ as follows 
\[ f(t) = 
   \begin{cases}
     f_L(t) & 0<t<t_1^* \hbox{ or } t_2^*<t<2 \\
     f_R(t) & t_1^*<t<t_2^*.
   \end{cases}
\]

We introduce a new quantity $Tr$. If the lower chamber is trapping, then we define in the upper chamber 
\begin{align*}
     Tr^U &= \left( \frac{1-f_1^-}{1-f_1^+} - a_1 \beta \right)\left( \frac{1-f_2^-}{1-f_2^+} - a_2 \alpha \right) + \left( \frac{1-f_1^+}{1-f_1^-} - a_1 \alpha \right)\left( \frac{1-f_2^+}{1-f_2^-} - a_2 \beta \right)\\
     &\ \ \ \ \ - a_1 a_2 \alpha \beta
\end{align*}
where $f_i^{\pm}=f(t_i^* \pm),$ $a_i=\dot{f}_i^-(1-f_i^+)-\dot{f}_i^+(1-f_i^-),$ 
$$\alpha=\left(\int_0^{t_1^*}+\int_{t_2^*}^2\right) \frac{ds}{(1-f(s))^2}\quad\text{and}\quad
\beta= \int_{t_1^*}^{t_2^*} \frac{ds}{(1-f(s))^2}.$$ 
If the upper chamber is trapping, then we define in the lower chamber 
\[
     Tr^L = \left( \frac{f_1^-}{f_1^+} - a_1' \beta' \right)\left( \frac{f_2^-}{f_2^+} - a_2' \alpha' \right) + \left( \frac{f_1^+}{f_1^-} - a_1' \alpha' \right)\left( \frac{f_2^+}{f_2^-} - a_2' \beta' \right) - a_1' a_2' \alpha' \beta'
\]
where $a_i'=\dot{f}_i^+ f_i^- - \dot{f}_i^- f_i^+,$ 
$\alpha'=\left(\int_0^{t_1^*}+\int_{t_2^*}^2\right) \frac{ds}{f(s)^2}$ and $\beta'= \int_{t_1^*}^{t_2^*} \frac{ds}{f(s)^2}$.

\begin{thm}\label{thm2}
 Assume $\lambda$ and $x_0$ are such that the relative positions of two slits change at two critical jumps and that $|Tr|>2$. Then there is $K, \zeta>0$ such that for any $\epsilon>0$, there exists $V_0=V_0(\epsilon)$ and $T=T(\epsilon)$ such that for each $V\geq V_0$ the complement of set $$\left\{(t_0,v_0): |v_0|\in [V, V+1]: \forall t\geq T\quad |v(t)|\geq \frac{|v(T)|}{K} e^{\zeta t}\; \right\} $$ has measure less than $\epsilon$, i.e. most orbits with initial energy $|v_0|>V_0$ start to accelerate exponentially after time $T$.
\end{thm}
The proof of this result is constructive. In particular, $T$ depends logarithmically on $\epsilon$
(see equations \eqref{ChooseKL} and \eqref{T-N}).

\begin{figure}[!ht]\label{fig2}
   \centering
   \caption{Trapping Regions}
   \includegraphics[width=6cm, height=6cm]{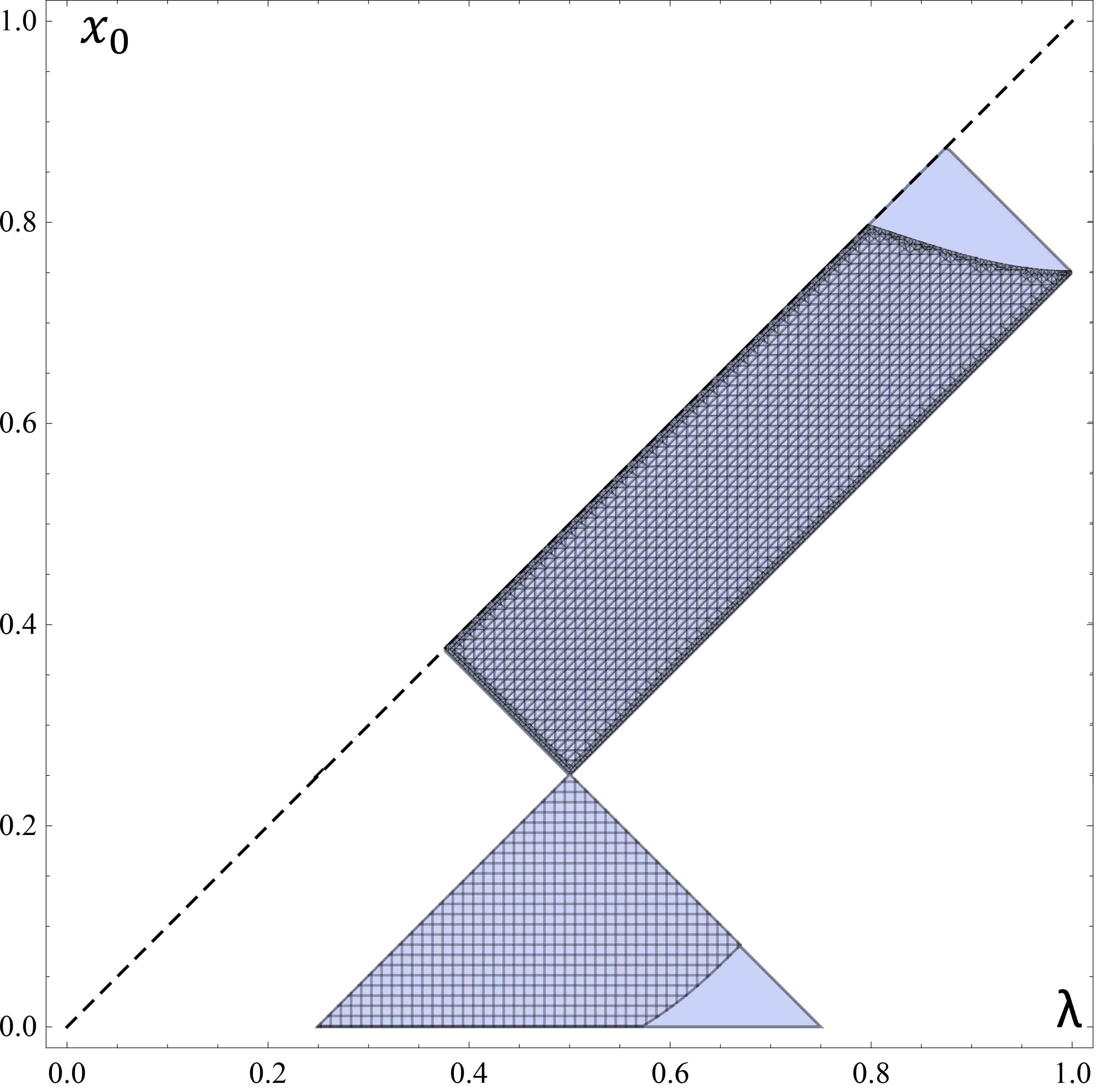}
   \floatfoot{The plot of trapping regions for $f_L(t)=0.3\cos(\pi t)+0.5$, $f_R(t)=0.3\sin(\pi t)+0.5$. The blue part indicates the values of $\lambda, x_0$ for which a trapping region exists: the upper rectangle is where the upper chamber is trapping and the lower triangle is where the lower chamber is trapping. The shaded part displays where the hyperbolicity assumption $|Tr|>2$ holds.}
\end{figure}

\begin{exm} To illustrate our results, we consider the case where
\begin{equation}
\label{ExFLR}
f_L(t)=0.3\cos(\pi t)+0.5,\quad f_R(t)=0.3\sin(\pi t)+0.5. 
\end{equation}
Then 
$$ \Delta(t):=f_L(t)-f_R(t)=0.3 \sqrt{2} \cos\left(\pi t+\frac{\pi}{4}\right). $$
A trapping region exists for $\lambda, x_0$ such that either 
$$\Delta(\lambda-x_0)>0,\quad  \Delta(2-\lambda-x_0)<0\quad\text{or}\quad 
\Delta(\lambda-x_0)<0,\quad  \Delta(2-\lambda-x_0)>0.$$
The former case is equivalent to $\lambda-x_0<0.25,\ 0.75<\lambda+x_0<1.75$; the upper chamber is trapping and the hyperbolicity assumption holds if $|Tr^L|>2$.\\ 
The latter case is equivalent to $\lambda-x_0>0.25,\ \lambda+x_0<0.75$;
the lower chamber is trapping and the hyperbolicity assumption holds if 
$|Tr^U|>2$.\\ 
Figure \ref{fig2} demonstrates that for $f_L, f_R$ defined above by \eqref{ExFLR}
the assumptions of Theorems \ref{thm1} and \ref{thm2} hold for a sizable set of parameters.
\end{exm}

The structure of the rest of the paper is the following.
In Section \ref{ScPrel} we describe the collision map.
In Section \ref{ScNF} we derive the normal form for the map obtained by considering the
next collision with the moving wall after the ball switches from left to right chamber or
{\it vice versa}. The proof of Theorem \ref{thm1} is given in Section \ref{ScTrap} and
the proof of Theorem \ref{thm2} is given in Section \ref{ScWait}.
In Section \ref{ScFinal} we summarize the tools developed
in the present paper and discuss open problems.

\section{Preliminaries}
\label{ScPrel}
Since the horizontal speed of the ball stays constant, only the vertical speed 
contributes to the energy change of the ball. 
This is why we only need to record the time $t$ and the vertical velocity $v$ 
immediately after each collision. Let us denote by $F$ the collision map.

For $i=1,2$, we denote as $R_i$ the strip in the $(t,v)$-plane bounded by the singularity line 
$\mathcal{S}_i = \{ t=t_i^* \}$ and its image $F\mathcal{S}_i.$ Also let $\tilde{R}_i = F^{-1}R_i$. 
We subdivide the singular strips $R$ into upper and lower chamber parts $R^+$ and $R^-$.\\

There are four possible scenarios when the ball makes a jump: the ball always hits the slits from above or below, the ball first hits from above then from below and vice versa.\\

We start with the easiest case when the ball always stays in the same chamber. Then the system is effectively equivalent to a Fermi-Ulam model with the motion (height) of the wall being the piecewise smooth 2-periodic function $f(t)$ with two jump discontinuities at $t_1^*$ and $t_2^*$.\\

Suppose that the ball is initially in the upper chamber. We omit the subscript $i$ as the formulas for passing through the two singularities are the same. If for $(t,v) \in \tilde{R}$ we have $f(t^*+)-f(t)<v(t^*-t)<2-f(t)-f(t^*)$, then the ball ends in the upper chamber after jumping and the model is equivalent to the one with a fixed ceiling and a moving floor (c.f. \autoref{fig3} on the left).\\
Two consecutive collisions $(t_n,v_n)$ and $(t_{n+1},v_{n+1})=F(t_n,v_n)$ satisfy 
\begin{equation}
 \left \{
  \begin{aligned}
    &v_{n+1} = v_n +2\dot{f}(t_{n+1})\\
    &2-f(t_n) -f(t_{n+1}) = v_n(t_{n+1}-t_n).
  \end{aligned}
 \right.\label{eq1}
\end{equation}

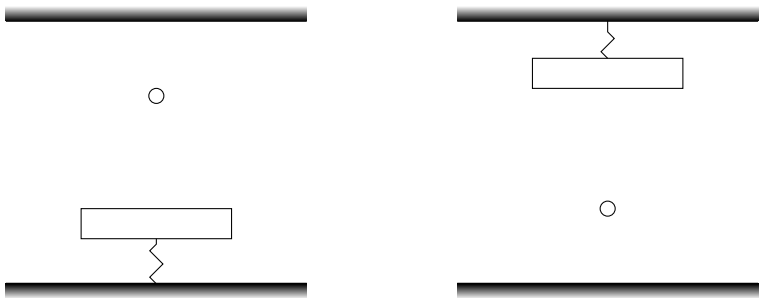
\begin{figure}[h!]
   \centering
      \begin{tikzpicture}
         \draw[thick] (0,0) -- (4,0);
         \draw[thick] (0,3.5) -- (4,3.5);
         \shade [top color=black,bottom color=white] (0,0) rectangle (4,-0.2);
         \shade [top color=white,bottom color=black] (0,3.5) rectangle (4,3.7);
         \draw (1,0.6) rectangle (3,1);
         \draw[snake=zigzag] (2,0) -- (2,0.6);
         \draw (2,2.5) circle (0.1cm);
         
         \draw[thick] (6,0) -- (10,0);
         \draw[thick] (6,3.5) -- (10,3.5);
         \shade [top color=black,bottom color=white] (6,0) rectangle (10,-0.2);
         \shade [top color=white,bottom color=black] (6,3.5) rectangle (10,3.7);
         \draw (7,2.6) rectangle (9,3);
         \draw[snake=zigzag] (8,3) -- (8,3.5);
         \draw (8,1) circle (0.1cm);
      \end{tikzpicture}
   \caption{Equivalent Fermi-Ulam Models for the Upper/Lower Chambers}\label{fig3}
\end{figure}

Similarly, suppose that the ball is initially in the lower chamber. If for $(t,v) \in \tilde{R}$ we have $f(t)-f(t^*+)<-v(t^*-t)<f(t)+f(t^*)$, then the ball ends in the lower chamber after jumping and the model is equivalent to the one with a moving ceiling and a fixed floor (c.f. \autoref{fig3} on the right).
Two consecutive collisions satisfy 
\begin{equation}
 \left \{
  \begin{aligned}
    &v_{n+1} = v_n +2\dot{f}(t_{n+1})\\
    &f(t_n) + f(t_{n+1}) = -v_n(t_{n+1}-t_n).
  \end{aligned}
 \right.\label{eq2}
\end{equation}

Now let us examine the switching cases.\\

Suppose that the ball is initially in the upper chamber and two consecutive collisions still follow \Cref{eq1} before the ball jumps from one slit to the other. However, when the ball jumps, if the next chamber is above the previous one when the ball passes through the singularities, then there is a possibility that the ball enters the lower chamber. More precisely, for $(t,v) \in \tilde{R}$, if $v(t_* -t) > 2-f(t)-f(t_*+)$, then the ball collides with 
the ceiling and then enters the lower chamber (c.f. \Cref{fig4} on the left); while if $v(t_* -t) < f(t_*+) - f(t)$, then the ball enters the lower chamber immediately after it leaves the previous slit (c.f. \Cref{fig4} on the right).\\

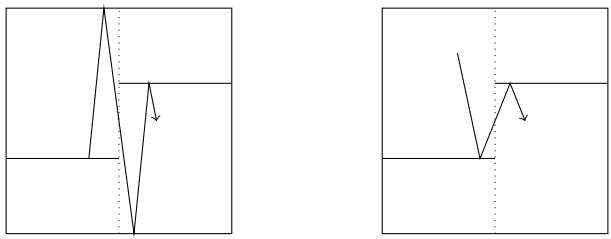
\begin{figure}[h!]
  \centering
     \begin{tikzpicture}
        \draw (0,0) rectangle (3,3)   (5,0) rectangle (8,3);
        \draw[dotted] (1.5,0)--(1.5,3) (6.5,0)--(6.5,3);
        \draw (0,1)--(1.5,1)   (1.5,2)--(3,2)   (5,1)--(6.5,1)   (6.5,2)--(8,2);
        \draw[->] (1.1,1)--(1.3,3)--(1.7,0)--(1.9,2)--(2,1.5);
        \draw[->] (6,2.4)--(6.3,1)--(6.7,2)--(6.9,1.5);
     \end{tikzpicture}
  \caption{From Upper to Lower Cases}\label{fig4}
\end{figure}

Two consecutive collisions satisfy the following \Cref{eq3} in the first case
\begin{equation}\label{eq3}
 \left\{ \begin{aligned}
  &v_n(t_{n+1}-t_n) =f(t_{n+1}) -f(t_n) +2\\
  &v_{n+1} = -v_n+2\dot{f}(t_{n+1})
  \end{aligned} \right.
\end{equation}

and the following \Cref{eq4} in the second case

\begin{equation}\label{eq4}
 \left\{ \begin{aligned}
  &v_n(t_{n+1}-t) = f(t_{n+1}) -f(t_n)\\
  &v_{n+1} = -v_n+2\dot{f}(t_{n+1}).
  \end{aligned} \right.
\end{equation}

On the other hand, suppose the ball is initially in the lower chamber and two consecutive collisions still follow \Cref{eq2} before the ball jumps from one slit to the other. When the ball jumps, if the next chamber is below the previous one when the ball passes through the singularities, then there is a possibility that the ball enters the upper chamber. More precisely, for $(t,v) \in \tilde{R}$, if $-v(t_* -t) > f(t)+f(t_*-)$, then the ball 
collides with the floor and enters the upper chamber (c.f. \Cref{fig5} on the left); 
while if $-v(t_* -t)<f(t) - f(t_*-)$, then the ball enters the lower chamber immediately 
after it leaves the previous slit (c.f. \Cref{fig5} on the right).\\

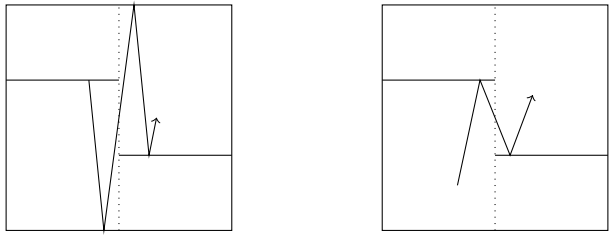
\begin{figure}[h!]
  \centering
     \begin{tikzpicture}
        \draw (0,0) rectangle (3,3)   (5,0) rectangle (8,3);
        \draw[dotted] (1.5,0)--(1.5,3) (6.5,0)--(6.5,3);
        \draw (0,2)--(1.5,2)   (1.5,1)--(3,1)   (5,2)--(6.5,2)   (6.5,1)--(8,1);
        \draw[->] (1.1,2)--(1.3,0)--(1.7,3)--(1.9,1)--(2,1.5);
        \draw[->] (6,0.6)--(6.3,2)--(6.7,1)--(7,1.8);
     \end{tikzpicture}
  \caption{From Lower to Upper Cases}\label{fig5}
\end{figure}

Two consecutive collisions satisfy the following \Cref{eq5} in the first case
\begin{equation}\label{eq5}
 \left\{ \begin{aligned}
  &v_n(t_{n+1}-t_n) =f(t_{n+1}) -f(t_n) -2\\
  &v_{n+1} = -v_n+2\dot{f}(t_{n+1})
  \end{aligned} \right.
\end{equation}

and the following \Cref{eq6} in the second case

\begin{equation}\label{eq6}
 \left\{ \begin{aligned}
  &v_n(t_{n+1}-t) = f(t_{n+1}) -f(t_n)\\
  &v_{n+1} = -v_n+2\dot{f}(t_{n+1}).
  \end{aligned} \right.
\end{equation}

\section{The Normal Form}
\label{ScNF}
In this section we study how the (vertical) velocity of the ball changes after one period $\Delta t=2$ given sufficiently large initial energy. We will first approximate the collision map $F$ with an action-angle coordinate away from singularities. Then we examine the collision dynamics when the ball passes through singularities.

\subsection{The Action-Angle Coordinate}
First we suppose that the ball collides with the slit from above and it does not make 
a jump at nearby collisions.\\
Let us denote $l(t)=1-f(t)$ and $\mathcal{L}_*=\int_0^2 l^{-2}(s) ds$.
\begin{lemma}\label{lemma41}
For $(t,v)\notin R_i \cup \tilde{R}_i$ $(i=1,2)$ and $v\gg 1$, there exists an action-angle coordinate $(\theta,I) = \Psi_U (t,v) \in \mathbb{R}/2\mathbb{Z} \times \mathbb{R}_+$ such that 
\[
\theta_{n+1} = \theta_n + \frac{2}{I_n} + \mathcal{O}\left( \frac{1}{I_n^4} \right), \ I_{n+1} = I_n + \mathcal{O}\left(\frac{1}{I_n^3}\right).
\]
In fact, $\displaystyle \theta =\theta(t) =\frac{2}{\mathcal{L}_*} \int_0^t \frac{ds}{l(s)^2} \mod 2$, $\displaystyle I = I(t,v) = \frac{\mathcal{L}_*}{2}\left(lv +l\dot{l} + \frac{l^2\ddot{l}}{3v}\right)$.
\end{lemma}
\begin{proof}
We can check the formula by a direct computation (c.f. Lemma 2.2 in \cite{deSD}), or we can derive it in an inductive way (c.f. Section 2.2 in \cite{Dol}). The basic idea is to find higher-order adiabatic invariants. For example, observe that 
$$v_{n+1}-v_n \approx -2\dot{l}(t_n), \ \ t_{n+1}-t_n \approx \frac{2l(t_n)}{v_n}.$$
This leads to the Euler scheme of the following ODE 
$$\frac{dv}{dt} = \frac{-v\dot{l}}{l}$$
which in turn gives us the zeroth order adiabatic invariant $I=lv$. Then we update the scheme by replacing $v$ with $I$ and look for the first order adiabatic invariant, etc. This scheme terminates at the second order adiabatic invariant $\displaystyle I=lv +l\dot{l} + \frac{l^2\ddot{l}}{3v}$.\\
Next, the formula for $\theta$ can be obtained reversely by solving the ODE $${\theta}'\frac{2l}{v}=\frac{2}{lv}$$
which leads to $\displaystyle \theta(t) = \int_0^t l^{-2}(s) ds$.\\
We observe that only the order $v$ term in $I$ is used to derive the formula for $\theta$ and it seems to produce an estimate only up to first order 
$$\theta_{n+1} - \theta_n = 2/I_n + \mathcal{O}(I_n^{-2}).$$ 
But in fact by noting the Taylor expansion of $l^{-2}$ and that
\[
t_{n+1} - t_n = \frac{2l(t_n)}{v_n} + \frac{2l(t_n)\dot{l}(t_n)}{v_n^2} + \frac{2l(t_n)\dot{l}(t_n)^2 + 2l(t_n)^2\ddot{l}(t_n)}{v_n^3} + \mathcal{O}(v_n^{-4})
\]
we obtain that 
\begin{align*}
\int_{t_n}^{t_{n+1}} l^{-2}(s) ds 
&= \frac{t_{n+1} - t_n}{l_n^2} - \frac{\dot{l}_n}{l_n^3}(t_{n+1} - t_n)^2 + (\frac{\dot{l}_n^2}{l_n^4}-\frac{\ddot{l}_n}{3l_n^3})(t_{n+1} - t_n)^3 + \mathcal{O}(v_n^{-4})\\
&= \frac{2}{l_nv_n} - \frac{2\dot{l}_n}{t_nv_n^2} +\frac{2\dot{l}_n^2-\frac{2}{3}l_n\ddot{l}_n}{l_nv_n^3} +\mathcal{O}(v_n^{-4})\\
&= \frac{2}{I_n} + \frac{2\dot{l}_n}{v_nI_n} +\frac{2l_n\ddot{l}_n}{3v_n^2I_n}- \frac{2\dot{l}_n}{l_nv_n^2} + \frac{2\dot{l}_n^2-\frac{2}{3}l_n\ddot{l}_n}{l_nv_n^3}+\mathcal{O}(v_n^{-4})\\
&=\frac{2}{I_n} + \frac{2\dot{l}_n^2}{v_n^2 I_n} + \frac{2l_n\ddot{l}_n}{3v_n^2 I_n} + \frac{2\dot{l}_n^2-\frac{2}{3}l_n\ddot{l}_n}{l_nv_n^3} + \mathcal{O}(v_n^{-4})\\
&=\frac{2}{I_n} + \mathcal{O}(v_n^{-4})
\end{align*}
where $l_n = l(t_n)$ and $\dot{l}_n = \dot{l}(t_n)$, which produces the desired third order estimate.\\

Finally, we need to rescale $\theta$ (and hence $I$) to make $\theta$ 2 periodic.
\end{proof}

Next we assume that the ball collides at the slits from below and it does not make a jump at nearby collisions.\\
We introduce a new function $g(t)=f(t)+1$. 
Then \Cref{eq2} becomes the same as \Cref{eq1} with $g$ in place of $f$ 
\begin{equation}
 \left \{
  \begin{aligned}
    &v_{n+1} = v_n +2\dot{g}(t_{n+1})\\
    &2-g(t_n) -g(t_{n+1}) = v_n(t_{n+1}-t_n)
  \end{aligned}
 \right.\label{eq7}
\end{equation}
Therefore all the computation above in Lemma \ref{lemma41} applies with $g$ in the place of $f$.\\
We define $m(t) = 1-g(t) = -f(t)$ and $\mathcal{M}_*= \int_0^2 m(s)^{-2} ds$.
We have an action-angle coordinate if the collision occurs in the lower chamber away from singularities 
\begin{lemma}\label{lemma42}
For $(t,v)\notin R_i \cup \tilde{R}_i$ $(i=1,2)$ and $v\ll -1$, there exists an action-angle coordinate $(\zeta,J) = \Psi_L (t,v) \in \mathbb{R}/2\mathbb{Z} \times \mathbb{R}_+$ such that 
\[
\zeta_{n+1} = \zeta_n + \frac{2}{J_n} + \mathcal{O}\left( \frac{1}{J_n^4} \right), \ J_{n+1} = J_n + \mathcal{O}\left(\frac{1}{J_n^3}\right).
\]
In fact, $\displaystyle \zeta =\zeta(t) =\frac{2}{\mathcal{M}_*} \int_0^t \frac{ds}{m(s)^2} \mod 2$, $\displaystyle J = J(t,v) = \frac{\mathcal{M}_*}{2}\left(mv +m\dot{m} + \frac{m^2\ddot{m}}{3v}\right)$.
\end{lemma}

\subsection{The Normal Forms}
In this section we present the Poincar\'e map $P$ from one singular strip to the other in four possible scenarios. We assume the initial energy of the ball is sufficiently large $|v_0| > V_*$ for some large $V_*$ in all the cases.

\subsubsection{The Upper-Upper Chamber Case}
We begin with the upper-upper chamber case, i.e. the ball stays in the upper chamber 
both before and after it makes a jump. Lemma \ref{lemma41} already depicts the dynamics away from singularities. Now let us scrutinize what occurs near the singularities $t_i^*$ ($i=1,2$) when the ball makes a jump.\\

\begin{figure}[!ht]
   \centering
      \begin{tikzpicture}
         \draw[->] (0,0)--(5,0) node[anchor=west]{$t$};
         \draw[->] (0,0)--(0,3) node[anchor=east]{$v$};
         \draw (0.6,0)--(1.3,2.5)  (1.5,0)--(1.5,2.5)  (2.4,0)--(1.7,2.5)  (2.6,0)--(3.3,2.5)  (3.5,0)--(3.5,2.5)  (4.4,0)--(3.7,2.5);
         \draw (1.25,2.5) node[anchor=south]{$\tilde{R}_1^+$}
               (1.75,2.5) node[anchor=south]{$R_1^+$}
               (3.25,2.5) node[anchor=south]{$\tilde{R}_2^+$}
               (3.75,2.5) node[anchor=south]{$R_2^+$}
               (1.5,0) node[anchor=north]{$t_1^*$}
               (3.5,0) node[anchor=north]{$t_2^*$}
               (0,0) node[anchor=east]{$V$};
         \filldraw (1.7,1.5) circle (1pt) node[anchor=north](1){$(t_1,v_1)$};
         \filldraw (3.3,1.9) circle (1pt) node[anchor=south](2){$(\tilde{t}_2,\tilde{v}_2)$};
         \filldraw (3.7,1.3) circle (1pt) node[anchor=west](3){$(\bar{t}_2,\bar{v}_2)$};
         \draw[->] (1) to [bend left] (2);
         \draw[->] (2) to [bend right] (3);
         \draw (2.5,2) node[anchor=north]{$F^{n_1}$};
      \end{tikzpicture}
   \caption{The Poincar\'e Map $P^{12}_{UU}$ on the Singular Strips}                          
   \end{figure}
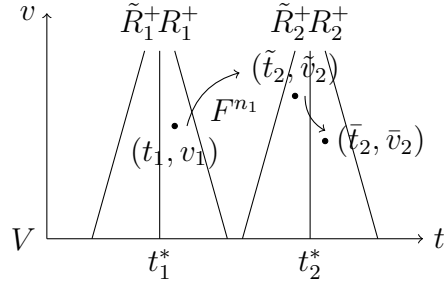

For $(t_1,v_1) \in R_1^+$ with $v_1 > V_*$, we denote $(\tilde{t}_2,\tilde{v}_2) =F^{n_1}(t_1,v_1) \in \tilde{R}_2^+$, where $\displaystyle n_1 = [\frac{I_1}{2}(\theta_2^* - \theta_1)]$ and $\displaystyle \theta_2^* = \frac{2}{\mathcal{L}_*}\int_0^{t_2^*} \frac{ds}{l(s)^2}$, and $(\bar{t}_2,\bar{v}_2) = F(\tilde{t}_2,\tilde{v}_2) \in R_2^+$. Similarly, for $(t_2,v_2) \in R_2^+$ with $v_2 \gg 1$, we denote $(\tilde{t}_1,\tilde{v}_1) =F^{n_2}(t_2,v_2) \in \tilde{R}_1^+$, where $\displaystyle n_2 = [\frac{I_2}{2}(2+\theta_1^* - \theta_2)]$ and $\displaystyle \theta_1^* = \frac{2}{\mathcal{L}_*}\int_0^{t_1^*} \frac{ds}{l(s)^2}$, and $(\bar{t}_1,\bar{v}_1) = F(\tilde{t}_1,\tilde{v}_1) \in R_1^+$.\\
We introduce a new pair of variables $(\tau, \mathcal{I})$ defined on the upper singular strips  
\[
\tau =
\begin{cases}
I(\theta - \theta_1^*) &\hbox{on $R_1^+$}\\
I(\theta - \theta_2^*) &\hbox{on $R_2^+$}
\end{cases},\  
\mathcal{I}=\frac{I}{\mathcal{L}_*} \hbox{  on $R_1^+$, $R_2^+$}
\]

Now we present the Poincar\'e maps $P^{12}_{UU}:R_1^+ \to R_2^+$ and $P^{21}_{UU}:R_2^+ \to R_1^+$ which captures the collision dynamics when the ball travels from one singular strips to the other. We need the following constants $(i=1,2)$:
\begin{align*}
&\Delta_i = \frac{1}{2} \frac{l(t_i^*+)}{l(t_i^*-)} \left( l(t_i^*-)\dot{l}(t_i^*+) - l(t_i^*+)\dot{l}(t_i^*-) \right),\\
&\Delta'_i = \frac{1}{8} l(t_i^*+)^2 \left( l(t_i^*-)\ddot{l}(t_i^*+) - l(t_i^*+)\ddot{l}(t_i^*-) \right),\\
&\Delta''_i = \frac{1}{24} l(t_i^*-)l(t_i^*+) \left( l(t_i^*-)\ddot{l}(t_i^*+) - l(t_i^*+)\ddot{l}(t_i^*-) \right).
\end{align*}

\begin{prop}[Upper-Upper]\label{prop43}
Suppose that $(\tau_1,\mathcal{I}_1) \in R_1^+$ and $\mathcal{I}_1>V_*$, and that 
$$f(t^*_2+)-f(t^*_2-) \lesssim l_2^- \{ \mathcal{L}_*\mathcal{I}_1(\theta_2^* - \theta_1^*) -\tau_1 \}_2 \lesssim 2-f(t^*_2+)-f(t^*_2-),$$ 
where $\lesssim$ means the inequality holds up to an error of order $\mathcal{O}(\frac{1}{\mathcal{I}})$, 
and $\{\bullet\}_2 = \bullet \mod 2$. Then the Poincar\'e map $P_{UU}^{12}: R_1^+ \to R_2^+$ is given by $(\bar{\tau}_2,\bar{\mathcal{I}}_2) = G_{UU}^{12}(\tau_1,\mathcal{I}_1) + H_{UU}^{12}(\tau_1,\mathcal{I}_1) + \mathcal{O}(\mathcal{I}_1^{-2})$ where 
\[ 
G_{UU}^{12}(\tau_1,\mathcal{I}_1) =\left( -\frac{l_2^-}{l_2^+} \{ \mathcal{L}_*\mathcal{I}_1(\theta_2^* - \theta_1^*) -\tau_1 \}_2 + 1+ \frac{l_2^-}{l_2^+}, \frac{l_2^+}{l_2^-} \mathcal{I}_1 + \Delta_2 (\bar{\tau}_2 -1) \right)
\]
and
\[ 
H_{UU}^{12}(\tau_1,\mathcal{I}_1) =\left(0, \Delta'_2 (\bar{\tau}_2 -1)^2/\mathcal{I}_1 + \Delta''_2/\mathcal{I}_1 \right)
\]
Similarly, suppose that $(\tau_2,\mathcal{I}_2) \in R_2^+$, $\mathcal{I}_2 > V_*$, and that 
$$f(t^*_1+)-f(t^*_1-) \lesssim l_1^- \{ \mathcal{L}_*\mathcal{I}_2(2+\theta_1^* - \theta_2^*) -\tau_2 \}_2  \lesssim 2-f(t^*_1+)-f(t^*_1-).$$ 
Then the Poincar\'e map $P_{UU}^{21}: R_2^+ \to R_1^+$ is given by 
$$(\bar{\tau}_1,\bar{\mathcal{I}}_1) = G_{UU}^{21}(\tau_2,\mathcal{I}_2) + H_{UU}^{21}(\tau_2,\mathcal{I}_2) + \mathcal{O}(\mathcal{I}_2^{-2})$$ 
where 
\[ 
G_{UU}^{21}(\tau_2,\mathcal{I}_2) =\left( -\frac{l_1^-}{l_1^+} \{ \mathcal{L}_*\mathcal{I}_2(2+\theta_1^* - \theta_2^*) -\tau_2 \}_2 + 1+ \frac{l_1^-}{l_1^+}, \frac{l_1^+}{l_1^-} \mathcal{I}_2 + \Delta_1 (\bar{\tau}_1 -1) \right),
\] 
\[ 
H_{UU}^{21}(\tau_2,\mathcal{I}_2) =\left(0, \Delta'_1 (\bar{\tau}_1 -1)^2/\mathcal{I}_2 + \Delta''_1/\mathcal{I}_2 \right)
\]
and $l_i^{\pm} = l(t_i^* \pm)$.
\end{prop}

\begin{proof}
We only derive the formula for $P_{UU}^{12}.$ The formula for $P_{UU}^{21}$ can be obtained in a similar fashion.\\

For the ease of notation we drop the sub/superscripts whenever they are clear from the context. 
Note that near the jump discontinuity at $t^*$
\begin{align*}
&l(\tilde{t}^*) = l_-(t^*) + \dot{l}_-(\tilde{t}-t^*) + \frac{1}{2}\ddot{l}_-(t^*)(\tilde{t}-t^*)^2 + \mathcal{O}(\tilde{v}^{-3}),\\
&l(\tilde{t}^*) = l_+(t^*) + \dot{l}_+(\bar{t}-t^*) + \frac{1}{2}\ddot{l}_+(t^*)(\bar{t}-t^*)^2 + \mathcal{O}(\tilde{v}^{-3}),\\
&\dot{l}(\tilde{t}^*) = \dot{l}_-(t^*) + \ddot{l}_-(\tilde{t}-t^*) + \mathcal{O}(\tilde{v}^{-2}),\\
&\dot{l}(\tilde{t}^*) = \dot{l}_+(t^*) + \ddot{l}_+(\bar{t}-t^*) + \mathcal{O}(\tilde{v}^{-2}),\\
&\ddot{l}(\tilde{t}^*) = \ddot{l}_-(t^*) +\mathcal{O}(\tilde{v}^{-1}),\\
&\ddot{l}(\tilde{t}^*) = \ddot{l}_+(t^*) +\mathcal{O}(\tilde{v}^{-1}),
\end{align*}
and that 
\[
\bar{v} = \tilde{v} - 2\dot{l}(\bar{t}), \quad \bar{t}-\tilde{t} = \frac{l(\bar{t})+l(\tilde{t})}{\bar{v}}.
\]
Hence by solving iteratively the implicit equation we attain
\[
\bar{t}-\tilde{t}=\frac{l_+ +l_-}{\tilde{v}} + (\dot{l}_+ +\dot{l}_-)\frac{\bar{t}-t^*}{\tilde{v}} - \frac{\dot{l}_-(l_+ +l_-)}{\tilde{v}^2} + \mathcal{O}(\tilde{v}^{-3}).
\]
By a straightforward but tedious computation we arrive at
\begin{align*}
2\mathcal{L}_*^{-1} \left( \frac{l_-}{l_+}\bar{I}-\tilde{I} \right) 
&= (l_+\dot{l}_- - l_-\dot{l}_+) - \frac{l_+\dot{l}_- - l_-\dot{l}_+}{l_+}(\bar{t} - t^*)\bar{v}\\
&\ \ \ \ +(l_+ \ddot{l}_- - l_- \ddot{l}_+ + \frac{l_-}{l_+}\dot{l}_+^2 - \dot{l}_-\dot{l}_+)(\bar{t} - t^*)\\
&\ \ \ \ +\left(\frac{l_-}{3}(l_+ \ddot{l}_+ - l_- \ddot{l}_-) + \frac{\ddot{l}_-}{2}(l_-^2 - l_+^2)\right)\frac{1}{\bar{v}}\\
&\ \ \ \ - \frac{l_+ \ddot{l}_- - l_- \ddot{l}_+}{2l_+}(\bar{t} - t^*)^2\bar{v} + \mathcal{O}(\tilde{v}^{-2})\\
&= \frac{l_-\dot{l}_+ - l_+\dot{l}_-}{l_+}\left( (\bar{t} - t^*)\bar{v}\big(1+\frac{\dot{l}_+}{\bar{v}}\big) - l_+ \right)\\
&\ \ \ \ + \frac{l_- \ddot{l}_+ - l_+ \ddot{l}_-}{2l_+ \bar{v}}\left( \big((\bar{t} - t^*)\bar{v}- l_+ \big)^2 + \frac{l_-l_+(l^- \ddot{l}_- - l_+ \ddot{l}_+)}{3(l_- \ddot{l}_+ - l_+ \ddot{l}_-)} \right) + \mathcal{O}(\tilde{v}^{-2}).
\end{align*}
It can be checked directly by Taylor expanding $\bar{I}$ and $\bar{\theta}$ that 
\[
\bar{\tau} = \bar{I}(\bar{\theta} - \theta_2^*) = \frac{1}{l_+} \left( (\bar{t}-t^*)\bar{v} + \dot{l}_+ (\bar{t}-t^*) \right) + \mathcal{O}(\tilde{v}^{-2}).
\]
Thus eventually we have
\[
\bar{\mathcal{I}} =  (l_+/l_-) \tilde{\mathcal{I}} + \Delta (\bar{\tau} -1) + \Delta' (\bar{\tau} -1)^2/\tilde{\mathcal{I}} + \Delta''/\tilde{\mathcal{I}} + \mathcal{O}(\tilde{\mathcal{I}}^{-2}).
\]
Now we compute $\bar{\tau}$. Observe that 
\[
\bar{\tau} = \frac{\bar{v} + \dot{l}_+ }{l_+}(\bar{t}-t^*) + \mathcal{O}(\bar{v}^{-2}),\ \tilde{I}(\tilde{\theta} - \tilde{\theta_2^*}) = \frac{\tilde{v} + \dot{l}_- }{l_-}(\tilde{t}-t^*) + \mathcal{O}(\bar{v}^{-2}).
\]
Therefore 
\begin{align*}
l_+ \bar{\tau} &= ((\bar{t}-\tilde{t}) + (\tilde{t}-t^*))(\tilde{v}+\dot{l}_- - (\dot{l}_- + \dot{l}_+)) + \mathcal{O}(\tilde{v}^{-2})\\
&= (\tilde{t}-t^*)(\tilde{v}+\dot{l}_-) + (\bar{t}-\tilde{t})(\tilde{v}+\dot{l}_-) - (\bar{t}-t^*)(\dot{l}_- + \dot{l}_+) + \mathcal{O}(\tilde{v}^{-2})\\
&= l_-\tilde{I}(\tilde{\theta} - \tilde{\theta_2^*}) 
+ \frac{l_- + l_+}{\tilde{v}} (\tilde{v}+\dot{l}_-) 
+(\dot{l}_- + \dot{l}_+)(\bar{t}-t^*) \\
&- \frac{\dot{l}_-(l_- + l_+)}{\tilde{v}} - (\bar{t}-t^*)(\dot{l}_- + \dot{l}_+) + \mathcal{O}(\tilde{v}^{-2})\\
&= l_-\tilde{I}(\tilde{\theta} - \tilde{\theta_2^*}) + l_- + l_+ + \mathcal{O}(\tilde{v}^{-2}),
\end{align*}
which gives 
$$\bar{\tau} = (l_-/l_+) \tilde{I}(\tilde{\theta} - \theta_2^*) + 1+ l_-/l_+ + \mathcal{O}(\tilde{I}^{-2}).$$
But \Cref{lemma41} implies that 
\[
\tilde{I} = I + \mathcal{O}(I^{-2}),\ \ \tilde{\theta} =\theta + \frac{2n_1}{I} + \mathcal{O}(I^{-3})
\]
hence we have 
\[
\tilde{I}(\tilde{\theta} - \theta_2^*)= \tau + 2n_1 + I(\theta_1^* - \theta_2^*) + \mathcal{O}(I^{-2})
\]
where $n_1=[\frac{I_1}{2}(\theta_2^*-\theta_1)]$.\\
We hitherto complete the proof of the formula for $P_{UU}^{12}$.
\end{proof}

\subsubsection{The Lower-Lower Chamber Case}
We present here the mirror case to Section 4.2.1, i.e. when the ball 
stays in the lower chamber both before and after it makes a jump.\\
We need the following constants $(i=1,2)$
\begin{align*}
\zeta_i^*&=\frac{2}{M_*}\int_0^{t_i^*} m(s)^{-2} ds\\
\Upsilon_i&=\frac{1}{2}\frac{m_i^+}{m_i^-}(m_i^-\dot{m}_i^+ - m_i^+\dot{m}_i^-)\\
\Upsilon_i'&=\frac{1}{8}m_i^{+2}(m_i^-\ddot{m}_i^+ - m_i^+\ddot{m}_i^-)\\
\Upsilon_i''&=\frac{1}{24}m_i^-m_i^+(m_i^-\ddot{m}_i^- - m_i^+\ddot{m}_i^+)
\end{align*}

We introduce a new pair of variables $(\rho,\mathcal{J})$ on the lower singular strips, which is the counterpart of $(\tau,\mathcal{I})$ as follows 
\[ 
\rho =
\begin{cases}
J(\zeta - \zeta_1^*) &\hbox{on $R_1^-$}\\
J(\zeta - \zeta_2^*) &\hbox{on $R_2^-$}
\end{cases} , \ 
\mathcal{J}=\frac{J}{\mathcal{M}_*} \hbox{ on $R_1^-$,$R_2^-$}
\]

\begin{prop}[Lower-Lower]\label{prop44}
Suppose that $(\rho_1,\mathcal{J}_1) \in R_1^-$, and $\mathcal{J}_1 > V_*$, and that 
$$f(t^*_2-)-f(t^*_2+) \lesssim -m_2^-\{ \mathcal{M}_*\mathcal{J}_1(\zeta_2^* - \zeta_1^*) -\rho_1 \}_2 \lesssim f(t^*_2-)+f(t^*_2+).$$ 
Then the Poincar\'e map $P_{LL}^{12} : R_1^- \to R_2^-$ is given by 
$$(\bar{\rho}_2,\bar{\mathcal{J}}_2) = G_{LL}^{12}(\rho_1,\mathcal{J}_1) + H_{LL}^{12}(\rho_1,\mathcal{J}_1)+ + \mathcal{O}(\mathcal{J}_1^{-2})$$ where
\[ 
G_{LL}^{12}(\rho_1,\mathcal{J}_1)=\left( -\frac{m_2^-}{m_2^+} \{ \mathcal{M}_*\mathcal{J}_1(\zeta_2^* - \zeta_1^*) -\rho_1 \}_2 + 1+ \frac{m_2^-}{m_2^+}, \frac{m_2^+}{m_2^-} \mathcal{J}_1 + \Upsilon_2 (\bar{\rho}_2 -1) \right)
\]
and 
\[ 
H_{LL}^{12}(\rho_1,\mathcal{J}_1) = \left(0, \Upsilon'_2 (\bar{\rho}_2 -1)^2/\mathcal{J}_1 + \Upsilon''_2/\mathcal{J}_1 \right).
\]
Similarly, suppose that $(\rho_2,\mathcal{J}_2) \in R_2^-$ and $\mathcal{J}_2 > V_*$, and that 
$$f(t^*_1-)-f(t^*_1+) \lesssim -m_1^-\{ \mathcal{M}_*\mathcal{J}_2(2+\zeta_1^* - \zeta_2^*) -\rho_2 \}_2 \lesssim f(t^*_1-)+f(t^*_1+).$$ 
Then the Poincar\'e map $P_{LL}^{21} : R_2^- \to R_1^-$ is given by 
$$(\bar{\rho}_1,\bar{\mathcal{J}}_1) = G_{LL}^{21}(\rho_2,\mathcal{J}_2) + H_{LL}^{21}(\rho_2,\mathcal{J}_2)+ + \mathcal{O}(\mathcal{J}_2^{-2})$$ 
where
\[ 
G_{LL}^{21}(\rho_2,\mathcal{J}_2)=\left( -\frac{m_1^-}{m_1^+} \{ \mathcal{M}_*\mathcal{J}_2(2+\zeta_1^* - \zeta_2^*) -\rho_2 \}_2 + 1+ \frac{m_1^-}{m_1^+}, \frac{m_1^+}{m_1^-} \mathcal{J}_2 + \Upsilon_1 (\bar{\rho}_1 -1) \right),
\]

\[ 
H_{LL}^{21}(\rho_2,\mathcal{J}_2) = \left(0, \Upsilon'_1 (\bar{\rho}_1 -1)^2/\mathcal{J}_2 + \Upsilon''_1/\mathcal{J}_2 \right),
\]
and $m_i^{\pm} = m(t_i^* \pm)$.
\end{prop}

\subsubsection{The Upper-Lower Chamber Case}
Now we suppose that $(\tilde{t},\tilde{v}) \in \tilde{R}$ and that the ball is in the upper chamber. Also we assume that the next wall is above the previous one when the ball passes through the singularity at $t=t_*$: $f(t_*-) < f(t_*+)$. Let $(\bar{t},\bar{v})=F(\tilde{t},\tilde{v})$.\\

If $v(t_* -t) > 2-f(t)-f(t_*+)$, then the ball collides 
with the ceiling and then enters the lower chamber (c.f. \Cref{fig4} on the left).\\
Rather than resorting to the detailed computation as we have done in the constant chamber cases, we insert an imaginary stationary slit, whose length is negligible, at the height $f_*= 1-\tilde{v}(t_*-\tilde{t})+l(\tilde{t})$, so that the two consecutive collisions at the moving slits are concatenated by two fictional collisions at the imaginary wall, to which the \nameref{prop43} and \nameref{prop44} formulas readily apply.\\
\begin{figure}[!h]
  \centering
    \begin{tikzpicture}
        \draw (0,0) rectangle (3,3);
        \draw[dotted] (1.5,0)--(1.5,3);
        \draw (0,1)--(1.5,1)   (1.5,2)--(3,2);
        \path[->] (1,1.5)--(1.1,1)  (1.9,2)--(2.1,1.3);
        \draw (1.1,1)--(1.3,3)--(1.7,0)--(1.9,2);
        \draw (1.1,1) node[anchor=north]{$(\tilde{t},\tilde{v})$}
              (1.9,2) node[anchor=south]{($\bar{t},\bar{v})$};
        \filldraw (1.5,1.5) circle (1pt) node[anchor=south east]{$(t_*,v_*)$};
    \end{tikzpicture}
  \caption{The Imaginary Stationary Wall}\label{fig7}
\end{figure}
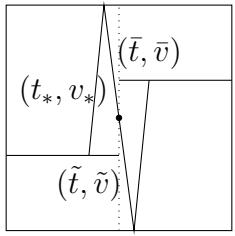

More precisely, as the ball leaves the previous slit at time $\tilde{t}$ with velocity $\tilde{v}$, it collides against the imaginary tiny slit at time $t_*$ and the outgoing velocity is still $v_*=\tilde{v}$ as the slit is stationary. Meanwhile we also imagine that the ball leaves from below the fictional slit at time $t_*$ with velocity $v_*=-\tilde{v}$ (with an abuse of notation), then it collides at the next moving slit at time $\bar{t}$ with outgoing velocity $\bar{v}$.\\
Let us denote $I_* = I(t_*,v_*)$, etc. We will need the following constants
\begin{align*}
   &\kappa_{Ii}= \frac{1}{2}m_+(\dot{m}_+ - m_+\dot{l}_-/l_-)\\
   &\kappa_{Ii}'= \frac{1}{2}m_+\dot{l}_-/l_-\\
   &\kappa_{Ii}''= \frac{1}{8}m_+\ddot{l}_-(1-\frac{1}{3}l_-^2)\\
   &\kappa_{Ii}'''= \frac{1}{4}m_+^2(\ddot{l}_- +\frac{1}{6}l_-\ddot{m}_+)\\
   &\kappa_{Ii}''''= \frac{1}{8}m_+^3\ddot{l}_-\\
   &\kappa_{Ii}'''''= \frac{1}{8}m_+^2\ddot{m}_+l_-
\end{align*}
where $i$ indicates that $l(t)$ and $m(t)$ are evaluated at $t=t_i^*$ ($i=1,2$).\\
Then the dynamics between the singular strips is captured by the following formula:

\begin{prop}[Upper-Lower I]\label{prop45}
Assume that $(\tau_1,\mathcal{I}_1) \in R_1^+$ with $\mathcal{I}_1 > V_*$ and that 
$$l_2^- \{ \mathcal{L}_*\mathcal{I}_1(\theta_2^* - \theta_1^*) -\tau_1 \}_2 \gtrsim 2-f(t^*_2+)-f(t^*_2-).$$
Then the Poincar\'e map $P_{UL\rm I}^{12}:R_1^+ \to R_2^-$ is given by 
$$(\bar{\rho}_2,\bar{\mathcal{J}}_2)=G_{UL\rm I}^{12}(\tau_1,\mathcal{I}_1) + H_{UL\rm I}^{12}(\tau_1,\mathcal{I}_1) + \mathcal{O}(\mathcal{I}_1^{-2})$$ where
\begin{align*} 
G_{UL\rm I}^{12}(\tau_1,\mathcal{I}_1)= \bigg( & \frac{l_2^-}{m_2^+}\{ \mathcal{L}_*\mathcal{I}_1(\theta_2^* - \theta_1^*) -\tau_1 \}_2 + \frac{m_2^+ - l_2^- -1}{m_2^+},\\
&-\frac{m_2^+}{l_2^-}\mathcal{I}_1 + \kappa_{I2}(\bar{\rho}_2-1) - \kappa_{I2}' \bigg)
\end{align*}
and
\[
H_{UL\rm I}^{12}(\tau_1,\mathcal{I}_1)=\left(0 , \frac{\kappa_{I2}''}{\mathcal{I}_1} +\kappa_{I2}'''\frac{\bar{\rho}_2-1}{I_1} +\kappa_{I2}''''\frac{(\bar{\rho}_2-1)^2}{\mathcal{I}_1} -\kappa_{I2}'''''\frac{(\bar{\rho}_2-1)^3}{\mathcal{I}_1} \right).
\]
Similarly, assume that $(\tau_2,\mathcal{I}_2) \in R_2^+$ with $\mathcal{I}_2 > V_*$ and that 
$$l_1^- \{ \mathcal{L}_*\mathcal{I}_2(2+\theta_1^* - \theta_2^*) -\tau_2 \}_2 \gtrsim 2-f(t^*_1+)-f(t^*_1-).$$ 
Then the Poincar\'e map $P_{UL\rm I}^{21}:R_2^+ \to R_1^-$ is given by 
$$(\bar{\rho}_1,\bar{\mathcal{J}}_1)=G_{UL\rm I}^{21}(\tau_2,\mathcal{I}_2) + H_{UL\rm I}^{21}(\tau_2,\mathcal{I}_2) + \mathcal{O}(\mathcal{I}_2^{-2})$$ where
\begin{align*} 
G_{UL\rm I}^{21}(\tau_2,\mathcal{I}_2)= \bigg( & \frac{l_1^-}{m_1^+}\{ \mathcal{L}_*\mathcal{I}_2(2+\theta_1^* - \theta_2^*) -\tau_2 \}_2 + \frac{m_1^+ - l_1^- -1}{m_1^+},\\
&-\frac{m_1^+}{l_1^-}\mathcal{I}_2 + \kappa_{I1}(\bar{\rho}_1-1) - \kappa_{I1}' \bigg)
\end{align*}
and
\[
H_{UL\rm I}^{21}(\tau_2,\mathcal{I}_2)=\left(0 , \frac{\kappa_{I1}''}{\mathcal{I}_2} +\kappa_{I1}'''\frac{\bar{\rho}_1-1}{\mathcal{I}_2} +\kappa_{I1}''''\frac{(\bar{\rho}_1-1)^2}{\mathcal{I}_2} -\kappa_{I1}'''''\frac{(\bar{\rho}_1-1)^3}{\mathcal{I}_2} \right).
\]
\end{prop}

\begin{proof}
We present the proof of $P_{UL\rm I}^{12}.$ 
The formula for $P_{UL\rm I}^{21}$ can be obtained similarly. We suppress the sub/superscripts whenever they are clear from the context.\\
We imagine that the ball collides at the fictional stationary wall at time $t_*$ with outgoing velocity $v_*=\tilde{v}$. Then $l_* = 1-f_* = \tilde{v}(t_*-\tilde{t})-l(\tilde{t})$ and we have 
\[
I_* = \frac{L_*}{2}l_*\tilde{v}, \ \tau_* = 0.
\]
From the \nameref{prop43} formula, 
\[
\tau_* = \frac{l_-}{l_*}\tilde{I}(\tilde{\theta}-\theta_2^*) + 1 + \frac{l_-}{l_*} + \mathcal{O}(v^{-2})
\]
\[
I_* = \frac{l_*}{l_-}\tilde{I} + \frac{L_*l_*^2}{2l_-}\dot{l}_- -\frac{L_*^2l_*^3\ddot{l}_-}{8\tilde{I}} + \frac{L_*^2l_*^2l_*\ddot{l}_-}{24\tilde{I}} + \mathcal{O}(v^{-2}).
\]
Now we imagine that the ball leaves from below the fictional stationary wall at time $t_*$ with outgoing velocity $v_*=-\tilde{v}$ (with an abuse of notation). Then $m_*= -m(\bar{t}) - v(\bar{t}-t_*)$ and we have 
\[
J_* = -\frac{M_*}{2}m_*\tilde{v}, \ \rho_*=0.
\]
From the \nameref{prop44} formula, 
\[
\bar{\rho} = 1+\frac{m_*}{m_+} + \mathcal{O}(v^{-2})
\]
\[
\bar{J} = \frac{m_+}{m_*}J_* + \frac{M_*}{2}m_+\dot{m}_+(\bar{\rho}-1) + \frac{M_*^2}{8J_*}m_+^2m_*\ddot{m}_+(\bar{\rho}-1)^2 - \frac{M_*^2}{24J_*}m_+^2m_*\ddot{m}_+ + \mathcal{O}(v^{-2}).
\]
We observe that 
\begin{align*}
 l_* &= v(t_*-\tilde{t})-l(\tilde{t})\\
     &= v(t_*-\tilde{t})- l_- - \dot{l}_-(\tilde{t}-t_*) + \mathcal{O}(v^{-2})\\
     &= -(\tilde{v}+\dot{l}_-)(t_*-\tilde{t})- l_- + \mathcal{O}(v^{-2})\\
     &=-l_-\tilde{I}(\tilde{\theta}-\theta_2^*) - l_- + \mathcal{O}(v^{-2})
\end{align*}
and that 
\begin{align*}
 m_* &= -m(\bar{t}) - v(\bar{t}-t_*)\\
     &= -m_+ - \dot{m}_+(\bar{t}-t_*) + (\bar{v}+2\dot{m}_+)(\bar{t}-t_*) + \mathcal{O}(v^{-2})\\
     &= -m_+ + (\bar{v}+\dot{m}_+)(\bar{t}-t_*) + \mathcal{O}(v^{-2})\\
     &= -m_+ +m_+\bar{\rho} + \mathcal{O}(v^{-2})
\end{align*}
Since $m_* = l_* -1$, 
\[
\bar{\rho} = -\frac{l_-}{m_+}\tilde{I}(\tilde{\theta}-\theta_2^*) + \frac{m_+ - l_- -1}{m_+} + \mathcal{O}(v^{-2}).
\]
Finally the relation between $I_*$ and $J_*$, together with \Cref{lemma41}, produces the formula for $\bar{J}$.
\end{proof}

If $v(t_* -t) < f(t_*+) - f(t)$, then the ball enters the lower chamber immediately after it leaves the previous slit (c.f. \Cref{fig3} on the right).\\
The imaginary wall trick no longer applies, so we have to return to the direct computation.\\
We will need the following constants
\begin{align*}
&\kappa_{\rm II i}''= \frac{1}{4}m_+^2\ddot{l}_-\\
&\kappa_{\rm II i}'''= \frac{1}{8}m_+^2(l_-\ddot{m}_+ - m_+\ddot{l}_-)\\
&\kappa_{\rm II i}''''= \frac{1}{24}m_+l_-(l_-^2\ddot{l}_- - l_-m_+\ddot{m}_+ - 3\ddot{l}_-)
\end{align*}
where $i$ indicates that $l(t)$ and $m(t)$ are evaluated at $t=t_i^*$ ($i=1,2$).

\begin{prop}[Upper-Lower \rm II]\label{prop46}
Assume that $(\tau_1,\mathcal{I}_1) \in R_1^+$ with $\mathcal{I}_1 > V_*$ and that 
$$f(t^*_2+)-f(t^*_2-) \gtrsim l_2^- \{ \mathcal{L}_*\mathcal{I}_1(\theta_2^* - \theta_1^*) -\tau_1 \}_2.$$ 
Then the Poincar\'e map $P_{UL\rm II}^{12}:R_1^+ \to R_2^-$ is given by 
$$(\bar{\rho}_2,\bar{\mathcal{J}}_2)=G_{UL\rm II}^{12}(\tau_1,\mathcal{I}_1) + H_{UL\rm II}^{12}(\tau_1,\mathcal{I}_1) + \mathcal{O}(\mathcal{I}_1^{-2})$$ where
\begin{align*}
G_{UL\rm II}^{12}(\tau_1,\mathcal{I}_1)=\bigg( &\frac{l_2^-}{m_2^+}\{ \mathcal{L}_*\mathcal{I}_1(\theta_2^* - \theta_1^*) -\tau_1 \}_2 + \frac{m_2^+ - l_2^- +1}{m_2^+},\\
&-\frac{m_2^+}{l_2^-}\mathcal{\mathcal{I}}_1 + \kappa_{\rm I 2}(\bar{\rho}_2-1) + \kappa_{\rm I 2}' \bigg)
\end{align*}
and
\[
H_{UL\rm II}^{12}(\tau_1,\mathcal{I}_1)=\left(0 , -\kappa_{\rm II 2}''\frac{\bar{\rho}_2-1}{\mathcal{I}_1} -\kappa_{\rm II 2}'''\frac{(\bar{\rho}_2-1)^2}{\mathcal{I}_1} -\frac{\kappa_{\rm II 2}''''}{\mathcal{I}_1} \right).
\]
Similarly, assume that $(\tau_2,\mathcal{I}_2) \in R_2^+$ with $\mathcal{I}_2 > V_*$ and that 
$$f(t^*_1+)-f(t^*_1-) \gtrsim l_1^- \{ \mathcal{L}_*\mathcal{I}_2(2+\theta_1^* - \theta_2^*) -\tau_2 \}_2.$$ 
Then the Poincar\'e map $P_{UL\rm II}^{21}:R_2^+ \to R_1^-$ is given by 
$$(\bar{\rho}_1,\bar{\mathcal{J}}_1)=G_{UL\rm II}^{21}(\tau_2,\mathcal{I}_2) + H_{UL\rm II}^{21}(\tau_2,\mathcal{I}_2) + \mathcal{O}(\mathcal{I}_2^{-2})$$ where
\begin{align*}
G_{UL\rm II}^{21}(\tau_2,\mathcal{I}_2)=\bigg( &\frac{l_1^-}{m_1^+}\{ \mathcal{L}_*\mathcal{I}_2(2+\theta_1^* - \theta_2^*) -\tau_2 \}_2 + \frac{m_1^+ - l_1^- +1}{m_1^+},\\
&-\frac{m_1^+}{l_1^-}\mathcal{I}_2 + \kappa_{\rm I 1}(\bar{\rho}_1-1) + \kappa_{\rm I 1}' \bigg)
\end{align*}
and
\[
H_{UL\rm II}^{21}(\tau_2,\mathcal{I}_2)=\left(0 , -\kappa_{\rm II 1}''\frac{\bar{\rho}_1-1}{\mathcal{I}_2} -\kappa_{\rm II 1}'''\frac{(\bar{\rho}_1-1)^2}{\mathcal{I}_2} -\frac{\kappa_{\rm II 1}''''}{\mathcal{I}_2} \right).
\]
\end{prop}

\begin{proof}
Again we only prove the formula for $P_{UL\rm II}^{12}$.\\
We have from \Cref{eq4} that
\begin{align*}
 v(\bar{t}-\tilde{t}) &= f(\bar{t}) -f(\tilde{t})\\
 \implies \tilde{v}( (\bar{t}-t_*) - (\tilde{t}-t_*) ) &= -m(\bar{t}) + l(\tilde{t}) -1\\
 \implies m_+ - (\bar{v}+\dot{m}_+)(\bar{t}-t_*) &= (\tilde{v}+\dot{l}_-)(\tilde{t}-t_*) + l_- -1 + \mathcal{O}(v^{-2})\\
 \implies m_+ - m_+ \bar{\rho} &= l_-\tilde{I}(\tilde{\theta}-\theta_2^*) +l_- -1  + \mathcal{O}(v^{-2})\\
 \implies \bar{\rho} &= -\frac{l_-}{m_+}\tilde{I}(\tilde{\theta}-\theta_2^*) + \frac{m_+ -l_- +1}{m_+} + \mathcal{O}(v^{-2})
\end{align*}
The computation is similar to that in the proof of Proposition \ref{prop43}. So we just list the key steps.\\
We observe that 
\begin{align*}
2\mathcal{M}_*^{-1}\bar{J} &= \bar{m}\bar{v} + \bar{m}\dot{\bar{m}}+\frac{\bar{m}^2\ddot{\bar{m}}}{3\bar{v}}\\
&=m_+\bar{v}+m_+\dot{m}_+ +m_+\dot{m}_+\bar{\rho} + m_+\ddot{m}_+(\bar{t}-t_*) + \frac{m_+^2\ddot{m}_+}{3\bar{v}}\\
&\ \ \ \ \ + \frac{\ddot{m}_+}{2}(\bar{t}-t_*)^2\bar{v} + \mathcal{O}(v^{-2})
\end{align*}
and that
\begin{align*}
2\mathcal{L}_*^{-1}\tilde{I} &= \tilde{l}\tilde{v} + \tilde{l}\dot{\tilde{l}}+\frac{\tilde{l}^2\ddot{\tilde{l}}}{3\tilde{v}}\\
&=l_- \tilde{v} + l_- \dot{l}_- + l_-\dot{l}_-\tilde{I}(\tilde{\theta}-\theta_2^*) + l_-\ddot{l}_-(\tilde{t}-t_*) + \frac{l_-^2\ddot{l}_-}{3\tilde{v}}\\
&\ \ \ \ \ + \frac{\ddot{l}_-}{2}(\tilde{t}-t_*)^2\tilde{v} + \mathcal{O}(v^{-2}).
\end{align*}
We also note from \Cref{eq4} that 
\[
\tilde{v}=-\bar{v}-2\dot{m}_+ -2\ddot{m}_+(\bar{t}-t_*) + \mathcal{O}(v^{-2})
\]
and that 
\begin{align*}
\bar{t}-\tilde{t} &= \frac{f(\bar{t})-f(\tilde{t})}{\tilde{v}} = \frac{-m(\bar{t}) + l(\tilde{t}) -1}{\tilde{v}}\\
&=\frac{-m_+ + l_- -1}{\tilde{v}} + (-\dot{m}_+ +\dot{l}_-)\frac{\bar{t}-t_*}{\tilde{v}} - \dot{l}_-\frac{\bar{t}-\tilde{t}}{\tilde{v}} + \mathcal{O}(v^{-3})\\
&=\frac{-m_+ + l_- -1}{\tilde{v}} + (-\dot{m}_+ +\dot{l}_-)\frac{\bar{t}-t_*}{\tilde{v}} + \dot{l}_-\frac{m_+ - l_- +1}{\tilde{v}^2} + \mathcal{O}(v^{-3}).
\end{align*}
Therefore 
\begin{align*}
2\mathcal{L}_*^{-1}\tilde{I} &= -l_-\bar{v}-2l_-\dot{m}_+ +m_+\dot{l}_- + m_+\dot{l}_-\bar{\rho} +\dot{l}_-\\
&\ \ \ \ \ + (m_+\ddot{l}_- -2l_-\ddot{m}_+ + \ddot{l}_-)(\bar{t}-t_*)- \frac{\ddot{l}_-}{2}(\bar{t}-t_*)^2\bar{v}\\
&\ \ \ \ \  -\frac{\ddot{l}_-}{2\bar{v}}\left((m_+ +1)^2 - \frac{l_-^2}{3} \right) +\mathcal{O}(v^{-2})
\end{align*}
hence 
\begin{align*}
\frac{2l_-}{\mathcal{M}_*}\bar{J} + \frac{2m_+}{\mathcal{L}_*}\tilde{I} 
&= m_+\dot{l}_- + m_+(l_-\dot{m}_+ -m_+\dot{l}_-)(\bar{\rho}-1) -\frac{1}{2}\mathcal{L}_*l_-m_+^2\ddot{l}_-\frac{\bar{\rho}-1}{I}\\
&\ \ \ \ \ -\frac{1}{4}\mathcal{L}_*l_-m_+^2(l_-\ddot{m}_+ - m_+\ddot{l}_-)\frac{(\bar{\rho}-1)^2}{I}\\
&\ \ \ \ \ -\frac{1}{12I}\mathcal{L}_*l_-m_+(l_-^2\ddot{l}_- - l_-m_+\ddot{m}_+ -3\ddot{l}_-) +\mathcal{O}(v^{-2})
\end{align*}
which produces the desired formula together with \Cref{lemma41}.
\end{proof}

\subsubsection{The Lower-Upper Chamber Case}
Finally we suppose that $(\tilde{t},\tilde{v}) \in \tilde{R}$ and that the ball is in the lower chamber. Also we assume that the next wall is below the previous one when the ball passes through the singularity at $t=t_*$: $f(t_*-) > f(t_*+)$. Again let $(\bar{t},\bar{v})=F(\tilde{t},\tilde{v})$.\\

If $-v(t_* -t) > f(t)+f(t_*-)$, then the ball collides
with the floor and enters the upper chamber (c.f. \Cref{fig4} on the left).\\
The imaginary stationary wall trick also applies in this case, which produces a desired formula with the following constants
\begin{align*}
   &\chi_{Ii}= \frac{1}{2}l_+(\dot{l}_+ - l_+\dot{m}_-/m_-)\\
   &\chi_{Ii}'= \frac{1}{2}l_+\dot{m}_-/m_-\\
   &\chi_{Ii}''= \frac{1}{8}l_+\ddot{m}_-(1-\frac{1}{3}m_-^2)\\
   &\chi_{Ii}'''= \frac{1}{4}l_+^2(\frac{1}{6}m_-\ddot{l}_+ -\ddot{m}_- )\\
   &\chi_{Ii}''''= \frac{1}{8}l_+^3\ddot{m}_-\\
   &\chi_{Ii}'''''= \frac{1}{8}l_+^2\ddot{l}_+m_-
\end{align*}
where $i$ indicates that $l(t)$ and $m(t)$ are evaluated at $t=t_i^*$ ($i=1,2$).

\begin{prop}[Lower-Upper I]\label{prop47}
Assume that $(\rho_1,\mathcal{J}_1) \in R_1^-$ with $\mathcal{J}_1 > V_*$ and that $-m_2^-\{ \mathcal{M}_*\mathcal{J}_1(\zeta_2^* - \zeta_1^*) -\rho_1 \}_2 \gtrsim f(t^*_2-)+f(t^*_2+)$. Then the Poincar\'e map 
$P_{LU\rm I}^{12}: R_1^- \to R_2^+$ is given by $$(\bar{\tau}_2,\bar{\mathcal{I}}_2)=G_{LU\rm I}^{12}(\rho_1,\mathcal{J}_1) + H_{LU\rm I}^{12}(\rho_1,\mathcal{J}_1) + \mathcal{O}(\mathcal{J}_1^{-2})$$ 
where 
\begin{align*} 
G_{LU\rm I}^{12}(\rho_1,\mathcal{J}_1)=\bigg( &\frac{m_2^-}{l_2^+}\{ \mathcal{M}_*\mathcal{J}_1(\zeta_2^* - \zeta_1^*) -\rho_1 \}_2 + \frac{l_2^+ - m_2^- +1}{l_2^+},\\
&-\frac{l_2^+}{m_2^-}\mathcal{J}_1 + \chi_{I2}(\bar{\tau}_2-1) + \chi_{I2}' \bigg)
\end{align*}
and
\[
H_{LU\rm I}^{12}(\rho_1,\mathcal{J}_1)=\left(0 , \frac{\chi_{I2}''}{\mathcal{J}_1} +\chi_{I2}'''\frac{\bar{\tau}_2-1}{\mathcal{J}_1} +\chi_{I2}''''\frac{(\bar{\tau}_2-1)^2}{\mathcal{J}_1} -\chi_{I2}'''''\frac{(\bar{\tau}_2-1)^3}{\mathcal{J}_1} \right).
\]
Similarly, assume that $(\rho_2,\mathcal{J}_2) \in R_2^-$ with $\mathcal{J}_2 > V_*$ and that 
$$-m_1^-\{ \mathcal{M}_*\mathcal{J}_2(2+\zeta_1^* - \zeta_2^*) -\rho_1 \}_2 \gtrsim f(t^*_1-)+f(t^*_1+).$$ 
Then the Poincar\'e map $P_{LU\rm I}^{21}: R_2^- \to R_1^+$ is given by 
$$(\bar{\tau}_1,\bar{\mathcal{I}}_1)=G_{LU\rm I}^{21}(\rho_2,\mathcal{J}_2) + H_{LU\rm I}^{21}(\rho_2,\mathcal{J}_2) + \mathcal{O}(\mathcal{J}_2^{-2})$$ 
where 
\begin{align*} 
G_{LU\rm I}^{21}(\rho_2,\mathcal{J}_2)=\bigg( &\frac{m_1^-}{l_1^+}\{ \mathcal{M}_*\mathcal{J}_2(2+\zeta_1^* - \zeta_2^*) -\rho_2 \}_2 + \frac{l_1^+ - m_1^- +1}{l_1^+},\\
&\frac{l_1^+}{m_1^-}\mathcal{J}_2 + \chi_{I1}(\bar{\tau}_1-1) + \chi_{I1}' \bigg)
\end{align*}
and
\[
H_{LU\rm I}^{21}(\rho_2,\mathcal{J}_2)=\left(0 , \frac{\chi_{I1}''}{\mathcal{J}_2} +\chi_{I1}'''\frac{\bar{\tau}_1-1}{\mathcal{J}_2} +\chi_{I1}''''\frac{(\bar{\tau}_1-1)^2}{\mathcal{J}_2} -\chi_{I1}'''''\frac{(\bar{\tau}_1-1)^3}{\mathcal{J}_2} \right).
\]
\end{prop}

Next, if $-v(t_* -t)<f(t) - f(t_*-)$, then the ball enters the lower chamber immediately after it leaves the previous slit (c.f. \Cref{fig5} on the right).\\
The computation as we have performed for Proposition \ref{prop46} can be reproduced here to present the formula in this case, the proof of which we ergo omit. We will need the following constants
\begin{align*}
&\chi_{\rm II i}''= \frac{1}{4}l_+^2\ddot{m}_-\\
&\chi_{\rm II i}'''= \frac{1}{8}l_+^2(m_-\ddot{l}_+ - l_+\ddot{m}_-)\\
&\chi_{\rm II i}''''= \frac{1}{24}l_+m_-(m_-^2\ddot{m}_- - m_-l_+\ddot{l}_+ - 3\ddot{m}_-)
\end{align*}
where $i$ indicates that $l(t)$ and $m(t)$ are evaluated at $t=t_i^*$ ($i=1,2$).

\begin{prop}[Lower-Upper \rm II]\label{prop48}
Assume that $(\rho_1,\mathcal{J}_1) \in R_1^-$ with $\mathcal{J}_1 > V_*$ and that $f(t^*_2-)-f(t^*_2+) \gtrsim -m_2^-\{ \mathcal{M}_*\mathcal{J}_1(\zeta_2^* - \zeta_1^*) -\rho_1 \}_2$. Then the Poincar\'e map 
$P_{LU\rm II}^{12}: R_1^- \to R_2^+$ is given by 
$$(\bar{\tau}_2,\bar{\mathcal{I}}_2)=G_{LU\rm II}^{12}(\rho_1,\mathcal{J}_1) + H_{LU\rm II}^{12}(\rho_1,\mathcal{J}_1) + \mathcal{O}(\mathcal{J}_1^{-2})$$ 
where
\begin{align*}
G_{LU\rm II}^{12}(\rho_1,\mathcal{J}_1)=\bigg( &\frac{m_2^-}{l_2^+} \{ \mathcal{M}_*\mathcal{J}_1(\zeta_2^* - \zeta_1^*) -\rho_1 \}_2 + \frac{l_2^+ - m_2^- -1}{l_2^+},\\
&-\frac{l_2^+}{m_2^-}\mathcal{J}_1 + \chi_{I2}(\bar{\tau}_2-1) - \chi_{I2}' \bigg)
\end{align*}
and
\[
H_{LU\rm II}^{12}(\rho_1,\mathcal{J}_1)=\left(0 , \chi_{\rm II 2}''\frac{\bar{\tau}_2-1}{\mathcal{J}_1} -\chi_{\rm II 2}'''\frac{(\bar{\tau}_2-1)^2}{\mathcal{J}_1} -\frac{\chi_{\rm II 2}''''}{\mathcal{J}_1} \right).
\]
Similarly, assume that $(\rho_2,\mathcal{J}_2) \in R_2^-$ with $\mathcal{J}_2 > V_*$ and that 
$$f(t^*_1-)-f(t^*_1+) \gtrsim -m_1^-\{ \mathcal{M}_*\mathcal{J}_2(2+\zeta_1^* - \zeta_2^*) -\rho_2 \}_2.$$ 
Then the Poincar\'e map $P_{LU\rm II}^{21}: R_2^- \to R_1^+$ is given by 
$$(\bar{\tau}_1,\bar{\mathcal{I}}_1)=G_{LU\rm II}^{21}(\rho_2,\mathcal{J}_2) + H_{LU\rm II}^{21}(\rho_2,\mathcal{J}_2) + \mathcal{O}(\mathcal{J}_2^{-2})$$ where
\begin{align*}
G_{LU\rm II}^{21}(\rho_2,\mathcal{J}_2)=\bigg( &\frac{m_1^-}{l_1^+} \{ \mathcal{J}_2(2+\zeta_1^* - \zeta_2^*) -\rho_2 \}_2 + \frac{l_1^+ - m_1^- -1}{l_1^+},\\
&-\frac{l_1^+}{m_1^-}\mathcal{J}_2 + \chi_{\rm I 1}(\bar{\tau}_1-1) - \chi_{\rm I 1}' \bigg)
\end{align*}
and
\[
H_{LU\rm II}^{21}(\rho_2,\mathcal{J}_2)=\left(0 , \chi_{\rm II 1}''\frac{\bar{\tau}_1-1}{\mathcal{J}_2} -\chi_{\rm II 1}'''\frac{(\bar{\tau}_1-1)^2}{\mathcal{J}_2} -\frac{\chi_{\rm II 1}''''}{\mathcal{J}_2} \right).
\]
\end{prop}

\section{Trapping Regions}
\label{ScTrap}
In this section, we present the proof of Theorem \ref{thm1}.  The assumptions in the theorem lead to the creation of a trapping region where the ball gains energy exponentially fast. \\
\begin{proof}
We choose $V_* \gg 1$ so that the normal forms in Section 4 hold for $|v|>V_*$. There are two cases.\\

(i) Suppose that $f_L(t_1^*)<f_R(t_1^*)$ and $f_L(t_2^*)>f_R(t_2^*)$. The relative positions of the two slits at two critical jumps trap the ball forever in the lower region once it enters. Henceforth Proposition \ref{prop44} predicts the change of energy after one period in the lower chamber to be
$$\bar{\bar{\mathcal{J}}} = \frac{m_1^+}{m_1^-}\frac{m_2^+}{m_2^-}\mathcal{J} + \mathcal{O}(1)$$
Furthermore, the relative positions of the slits at two critical times guarantee that $m(t_1^*+)<m(t_1^*-)<0$, $m(t_2^*+)<m(t_2^*-)<0$, so the energy of the ball grows exponentially fast at rate $\frac{m_1^+}{m_1^-}\frac{m_2^+}{m_2^-}>1$ in the lower chamber.\\

\begin{figure}[!ht]
    \centering
        \begin{tikzpicture}
           \draw (0,0) rectangle (2,2)   (3,0) rectangle (5,2);
           \draw[dotted] (1,0)--(1,2)   (4,0)--(4,2);
           \draw (0,0.5)--(1,0.5)   (1,1.5)--(2,1.5)   (3,1.5)--(4,1.5)   (4,0.5)--(5,0.5);
           \draw (1,0) node[anchor=north]{$t_1^*$}
                 (4,0) node[anchor=north]{$t_2^*$};
           \draw[->] (0.9,0.5)--(1.1,0)--(1.3,1.5)--(1.4,1.2);
           \draw[<-] (3.7,1.2)--(3.8,1.5)--(3.95,0)--(4.1,0.5);
        \end{tikzpicture}
    \caption{The Trapping Lower Chamber}
    \label{fig8}
\end{figure}
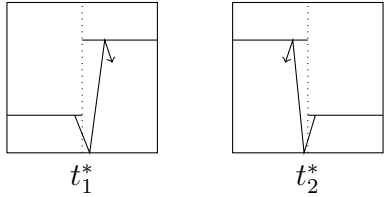

If the ball starts from the lower chamber with $v_0<-V_*$, it enjoys exponential energy growth with time immediately.\\
If the ball starts from the upper chamber with $v_0>V_*$, by Proposition \ref{prop43} and the relative positions of the slits, its energy decreases at an exponential rate $\frac{l_1^+}{l_1^-}\frac{l_2^+}{l_2^-}<1$ until it either enters the lower chamber or it enter the low energy region $|v|<V_*$ and
the normal form no longer applies. 
 For any $V>V_*$, we denote 
$$\mathcal{U}_V=\{(t_0,v_0): V < v_0 < V+1, \ \limsup v_n<V_*\}$$
We claim that $\mes(\mathcal{U}_V)=0$. Otherwise we note that 
$\mathcal{U}_V \subseteq \{|v_n|<V+1,\ \forall n>0\}$, which is bounded and invariant.
Hence, the Poincar\'e recurrence theorem implies that almost every point in $\mathcal{U}_V$ would return infinitely often to energy level $|v_n| > V$, which is impossible as it would contradict the definition of
$\mathcal{U}_V.$ Our claim implies that almost all points in the energy shell 
$\mathcal{W}_V=\{V<v<V+1\}$ eventually return to high energy level $|v_n|>V_*$, which is only made possible if the ball enters the lower chamber and the foregoing discussion ensures exponential energy growth afterwards.\\

(ii) Suppose that $f_L(t_1^*)>f_R(t_1^*)$ and $f_L(t_2^*)<f_R(t_2^*)$. Now the relative positions of the two slits at two critical jumps indicate that the upper region is trapping and then Proposition \ref{prop43} guarantees an exponential energy gain at rate $\frac{l_1^+}{l_1^-}\frac{l_2^+}{l_2^-}>1$ in the upper chamber once the ball gets trapped. The rest of the analysis is similar to Case (i).
\end{proof}

\begin{exm}
 In general it is not possible to improve the result that non-escaping orbits have finite measure to one with zero measure. For example, we start with $\tilde{f}_L(t)=\tilde{f}_R(t)=a\cos 4\pi t +0.5$ for some small $a>0$. We take $x_0=0, \lambda=0.5$, so $t_1^*=0.5, t_2^*=1.5$. We consider a 4-periodic orbit $\mathcal{P}$ starting at $t_0=0.25, v_0=2+4a$. 
Then
$$ (t_1, v_1)=(0.75,2+4a), \quad 
(t_2, v_2)=(1.25,2+4a), \quad 
(t_3, v_3)=(1.75,2+4a), \quad . $$
 
 We can slightly modify $\tilde{f}_L$ near $t_1^*, t_2^*$ in such a way that 
 $$f_L(0.5)>f_R(0.5), \quad f_L(1.5)<f_R(1.5)$$ 
 so that the upper chamber is trapping and that the periodic orbit $\mathcal{P}$ does not see this modification. Observe that $\mathcal{P}$ is elliptic for all small $a$ as the trace of the collision map $F$ along $\mathcal{P}$ is $tr (dF_{\mathcal{P}}) = 2-\frac{8a\pi^2}{1+2a} \in (0,2)$ for $0<a<\frac{1}{4\pi^2-2}$. Now the matrix $dF_{\mathcal{P}}$ is conjugate to a rotation by $2\pi\alpha$ with $\cos 2\pi\alpha=1-\frac{4a\pi^2}{1+2a}$. We can easily choose $a$ such that the rotation angle $\alpha$ is Diophantine, then Herman's Last Geometric Theorem guarantees the stability of the elliptic orbit $\mathcal{P}$, i.e. there exists an elliptic island of bounded trajectories around $\mathcal{P}$ (c.f. \cite[Theorem 4]{FK}).
\end{exm}

Although the assumptions of Theorem \ref{thm1} are compatible with existence of a positive measure
set of bounded orbits,
we can eliminate the possibility of oscillatory orbits. Recall that a (forward) oscillatory orbit is 
an orbit such that 
$$\limsup_{t\to+\infty} |v(t)| =\infty \quad \text{and} \quad \liminf_{t\to+\infty} |v(t)| <\infty. $$

\begin{cor}
 In presence of a trapping region oscillatory orbits do not exist.
\end{cor}

\begin{proof}
We may assume without loss of generality that the lower chamber is trapping. All the high energy orbits in the lower chamber gain energy exponentially immediately.\\
Now suppose that the ball is in the upper chamber and arrives at high energy level at some $v>V_*$, then it decelerates exponentially as observed in the proof of Theorem \ref{thm1} until it hits the lower chambers or the normal form no longer applies. In either case, it either starts to accelerate exponentially or remains in the low energy region $|v|<V_*$ afterwards.
\end{proof}

\section{Waiting Time for Exponential Acceleration}
\label{ScWait}
In this section we show that in the presence of the trapping region, the majority of orbits
with sufficiently high energy get trapped quickly under the hyperbolicity assumption. 
Throughout this section we assume without loss of generality that the lower chamber is trapping and $|Tr|>2$. The quantity $Tr$ is in fact the trace of the derivative of the linear map $G_U=G_{UU}^{21} \circ G_{UU}^{12}$ and the hyperbolicity assumption $|Tr|>2$ indicates that $G_U$ is hyperbolic.

\subsection{Almost Sure Escape for the Limiting Map}
We first restrict ourselves to the linear parts $G_{UU}$'s of the dynamics in Proposition \ref{prop43}, 
which approximates $P_{UU}$'s with an error of order $\mathcal{O}(I^{-1})$ when the velocity is large $v>V_*$.\\
We note that 
\begin{align*}
    (\bar{\tau}_2,\bar{I}_2)&=G_{UU}^{12}(\tau_1,I_1)\\
    &=\left( -\frac{l_2^-}{l_2^+} \{ \mathcal{L}_*\mathcal{I}_1(\theta_2^* - \theta_1^*) -\tau_1 \}_2 + 1+ \frac{l_2^-}{l_2^+}, \frac{l_2^+}{l_2^-} \mathcal{I}_1 + \Delta_2 (\bar{\tau}_2 -1) \right)
\end{align*}
if $1-\frac{l_2^+}{l_2^-} < \{ \mathcal{L}_*\mathcal{I}_1(\theta_2^* - \theta_1^*) -\tau_1 \}_2 < 1+\frac{l_2^+}{l_2^-}$. The boundary lines 
\begin{equation}
\{ \mathcal{L}_*\mathcal{I}_1(\theta_2^* - \theta_1^*) -\tau_1 \}_2=1-\frac{l_2^+}{l_2^-}, \quad \{ \mathcal{L}_*\mathcal{I}_1(\theta_2^* - \theta_1^*) -\tau_1 \}_2=1+\frac{l_2^+}{l_2^-}
\end{equation}
cut out from $R_1^+$ a sequence of boxes 
$$A_n=\{1-\frac{l_2^+}{l_2^-} +2n < \mathcal{L}_*\mathcal{I}_1(\theta_2^* - \theta_1^*) -\tau_1 < 1+\frac{l_2^+}{l_2^-}+2n\}$$ 
whose points will remain in the upper chamber under $G_{UU}^{12}$, while the other points 
will enter the lower chamber, when jumping from right to left at $t_2^*$.\\
We also observe that 
\begin{align*}
    (\bar{\tau}_1,\bar{I}_1)&=G_{UU}^{21}(\tau_2,I_2)\\
    &=\left( -\frac{l_1^-}{l_1^+} \{ \mathcal{L}_*\mathcal{I}_2(2+\theta_1^* - \theta_2^*) -\tau_2 \}_2 + 1+ \frac{l_1^-}{l_1^+}, \frac{l_1^+}{l_1^-} \mathcal{I}_2 + \Delta_1 (\bar{\tau}_1 -1) \right)
\end{align*}
if $1-\frac{l_1^+}{l_1^-} < \{ \mathcal{L}_*\mathcal{I}_2(2+\theta_1^* - \theta_2^*) -\tau_2 \}_2 < 1+\frac{l_1^+}{l_1^-}$ and that the boundary lines 
\begin{equation}
\label{BBoundaries}
\{ \mathcal{L}_*\mathcal{I}_2(2+\theta_1^* - \theta_2^*) -\tau_2 \}_2=1-\frac{l_1^+}{l_1^-}, \quad \{ \mathcal{L}_*\mathcal{I}_2(2+\theta_1^* - \theta_2^*) -\tau_2 \}_2=1+\frac{l_1^+}{l_1^-}
\end{equation}
cut out from $R_2^+$ another sequence of boxes 
$$B_n=\{1-\frac{l_1^+}{l_1^-}+2n<\mathcal{L}_*\mathcal{I}_2(2+\theta_1^* - \theta_2^*) -\tau_2<1+\frac{l_1^+}{l_1^-}+2n\}$$ 
whose points will remain in the upper chamber under $G_{UU}^{21}$, while the points outside will enter the lower chamber, when jumping from left to right at $t_1^*$.\\

We define $G_U=G_{UU}^{21} \circ G_{UU}^{12}$ on $R_1^+$. 
Both $G_{UU}^{12}$ and $G_{UU}^{21}$ are piecewise affine maps, and 
the derivative of $G_U$ part is a constant matrix $DG_U=DG_{UU}^{21} \cdot DG_{UU}^{12}$ where
\[ DG_{UU}^{12} =
   \begin{pmatrix}
    \frac{l_2^-}{l_2^+} & -\frac{l_2^-}{l_2^+} \mathcal{L}_* (\theta_2^* - \theta_1^*)\\
    \Delta_2 \frac{l_2^-}{l_2^+} & -\Delta_2 \frac{l_2^-}{l_2^+} \mathcal{L}_* (\theta_2^* - \theta_1^*) + \frac{l_2^+}{l_2^-}
   \end{pmatrix},
\]

\[ DG_{UU}^{21} =
   \begin{pmatrix}
    \frac{l_1^-}{l_1^+} & -\frac{l_1^-}{l_1^+} \mathcal{L}_* (2+\theta_1^* - \theta_2^*)\\
    \Delta_1 \frac{l_1^-}{l_1^+} & -\Delta_1 \frac{l_1^-}{l_1^+} \mathcal{L}_* (2+\theta_1^* - \theta_2^*) + \frac{l_1^+}{l_1^-}
   \end{pmatrix}
\]

Since $\det(DG_U)=1$ and $|Tr(DG_U)|>2$, it has unstable eigenvalue $\Lambda_u$ with unstable eigenvector $\bm{e}_u$, and stable eigenvalue $\Lambda_s $ with stable eigenvector $\bm{e}_s$.\\

We observe that each box $A$ is foliated by unstable lines.\\
We say that an unstable line $\gamma$ in a box $A$ is \textit{good} if it breaks after one period and at least two components remain in the upper chamber, otherwise we say it is \textit{bad}.\\
A good line is good as a solid part of it enters the trapping region after one period under the linear map $G_U$:
\begin{lemma}\label{lemma61}
  Let $\gamma$ be a good unstable line in some box $A$. Then the proportion of points on $\gamma$ which remain in the upper chamber after one period is at most 
 $$D=\frac{1+2\frac{l_1^+}{l_1^-}}{2+\frac{l_1^+}{l_1^-}}<1. $$
\end{lemma}
\begin{proof}
We first note that $G_{UU}^{12}(\gamma)$ remains a complete piece in $R_2$ as $\gamma$ lies in $A$ and that $G_{UU}^{12}$ maps the boundaries of $A$ into two vertical lines 
$$G_{UU}^{12} \left(\{ \mathcal{L}_*\mathcal{I}_1(\theta_2^* - \theta_1^*) -\tau_1 \}_2=1-\frac{l_2^+}{l_2^-} \right) \subseteq \{\tau_2=2\},$$ 
$$G_{UU}^{12} \left(\{ \mathcal{L}_*\mathcal{I}_1(\theta_2^* - \theta_1^*) -\tau_1 \}_2=1+\frac{l_2^+}{l_2^-} \right) \subseteq \{\tau_2=0\}.$$

$G_{UU}^{12}(\gamma)$ has to stretch across at least two $B$-boxes if $\gamma$ has at least two pieces remaining in the upper chamber after one period.
\begin{figure}[!ht]
    \centering
       \begin{tikzpicture}
          \draw (0,0)--(2,1)--(2,1.5)--(0,0.5)--(0,0)
                (0,0.8)--(2,1.8)--(2,2.3)--(0,1.3)--(0,0.8)
                (0,3)--(2,4)--(2,3.5)--(0,2.5)--(0,3);
          \draw[red] (1.7,0.85)--(0,3);
          \draw[dotted] (1.3,2.2)--(1.3,2.8);
          \draw[<->] (-0.1,0.5)--(-0.1,1.3) node[midway,left]{2};
          \draw[<->] (0.08,0.8)--(0.08,1.3) node[midway,right]{$\frac{2l_1^+}{l_1^-}$};
       \end{tikzpicture}
    \caption{A good curve partly enters the trapping region}
    \label{fig9}
\end{figure}
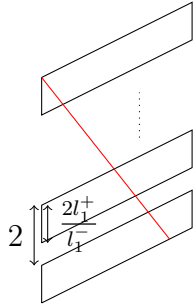

Suppose that $G_{UU}^{12}(\gamma)$ stretches across $N$ $B$-boxes for some $N>1$. 
It is easy to see that for a fixed $N$, the highest proportion of points staying in the upper chamber is achieved when $G_{UU}^{12}(\gamma)$ ends on the boundaries of the top and bottom boxes as shown in Figure \ref{fig9}. However the height of a $B$-box is equal to $2\frac{l_1^+}{l_1^-}$ (c.f. \eqref{BBoundaries}) while the height of the fundamental domain is equal to 2 (c.f. Figure \ref{fig9}). This implies that in the optimal situation $\frac{2(1-\frac{l_1^+}{l_1^-})(N-1)}{2N+2\frac{l_1^+}{l_1^-}}$ of the points on $G_{UU}^{12}(\gamma)$ land in the lower chamber after jumping from left to right at $t_1^*$. This proportion is larger than $\frac{(1-\frac{l_1^+}{l_1^-})(2-1)}{2+\frac{l_1^+}{l_1^-}}$ as it is an increasing function in $N$ and $N\geq 2$. Then the largest portion which remains in the upper chamber is given by 
$D=\frac{1+2\frac{l_1^+}{l_1^-}}{2+\frac{l_1^+}{l_1^-}}.$ 
Recall that the relative positions of two slits at $t_1^*$ implies that $\frac{l_1^+}{l_1^-}<1$, so $D<1$.
\end{proof}

Next we need to control the number of the short bad pieces as an unstable line breaks under the iterations of the linear map $G_U$.\\
Suppose $\gamma$ is an unstable line in some box $A$. For $x \in \gamma$, we denote as $r_n(x)$ the distance from $x_n$ to the nearest boundary of the component $\gamma_n$ containing $x_n$. Employing the argument in Section 5 of \cite{Dol} we obtain the following Growth Lemma.
\begin{lemma}[Growth Lemma]\label{lemma62}
  There exists a constant $C^*$ s.t. for any small $\epsilon>0$ and any $n\in\mathbb{N}$ 
  $$\mes_{\gamma}\{x\in\gamma: r_n(x)<\epsilon\} \le C^* \epsilon$$
\end{lemma}
\begin{proof}
 Let $k_n(\delta)$ denote the max number of the pieces that an unstable line of length less than $\delta$ can be cut into. We define $k_n=\lim_{\delta \to 0} k_n(\delta)$. We claim that $k_n \le 8n$. Indeed since the singularities of $G_U^n$ are lines and there are at most $8n$ possibilities for slopes. Consequently, there exists $\delta_0$ so small that $k_n(\delta) \le 16n$ for any $\delta<\delta_0$. We choose $n_0$ such that $\frac{32n_0}{\Lambda_u^{n_0}}<1$ and by replacing $G_U$ with $G_U^{n_0}$ we can always assume $n_0=1$.\\
 For inductive purposes we cut a long unstable line into pieces shorter than $\delta_0$ and let $\bar{r}_n(x)$ denote the distance from $x_n$ to the nearest real or artificial boundary of the component containing $x_n$. We note that by doing so we improve the estimate as $\bar{r}_n(x) \le r_n(x)$ and it suffices to prove the statement for $\bar{r}_n$.\\
 First we observe that 
\begin{equation}
\label{R0Bound}
\mes_{\gamma} \{ \bar{r}_0(x)<\epsilon \} \le \frac{2L}{\delta_0}\epsilon .
\end{equation}
 $\bar{r}_{n+1}(x)$ is less than $\epsilon$ if $x_{n+1}$ either passes a real or artificial singularity, where $L$ is the unstable height of $\gamma$. The former is controlled by 
 $2k_1 (\delta_0) \mes_{\gamma} \{ r_n<\frac{\epsilon}{\Lambda_u} \}$ while the latter by $2k_1 (\delta_0)\frac{L}{\delta_0}\epsilon$. Therefore 
 $$\mes_{\gamma} \{ \bar{r}_{n+1}<\epsilon \} \le \frac{32}{\Lambda_u} 
 \mes_{\gamma} \{ r_n<\epsilon \} + 32\frac{L}{\delta_0}\epsilon.$$
 Thus by induction we conclude that 
 $$\mes_{\gamma} \{ \bar{r}_n<\epsilon \} \le \left(\frac{32}{\Lambda_u}\right)^n \mes_{\gamma} \{ \bar{r}_0(x)<\epsilon \} + 32\frac{L}{\delta_0}\epsilon \left(1+\cdots +\left(\frac{32}{\Lambda_u}\right)^{n-1} \right).$$
 Since $\frac{32}{\Lambda_u}<1,$ \eqref{R0Bound}
gives the desired growth control with 
$$ C^*=\left(\frac{32}{\Lambda_u}\right)\frac{2L}{\delta_0}+\frac{32L}{\delta_0(1-\frac{32}{\Lambda_u})}.
\qedhere
$$   
\end{proof}

Finally we show that under the linear approximation map $G_U$ almost every point will eventually escape to the trapping region:
\begin{prop}\label{prop63}
   In each box $A$, for any $\epsilon>0$, there exists $N=N(\epsilon)$ such that all but an $\epsilon$-measure set of points in $A$ enter the lower chamber within $N$ periods. In particular, almost every point will leave the upper chamber in the future.
\end{prop}

\begin{proof}
Fix $\epsilon>0$. Choose $k,l$ such that 
$$D^k L< 0.5 \epsilon\quad\text{and} \quad (kl+1) \frac{C^*+L^2}{\Lambda_u^{l/2}} < 0.5 \epsilon,$$ and take $N=kl+1$.\\
We suppose that under the linear map $G_U$ a point $x$ on a unstable line $\gamma$ stays in the upper chamber up to $N$ periods.\\
If the trajectory of $x$ lands on good lines more than $k$ times in $N$ periods, then Lemma \ref{lemma61} shows that for good lines the portion which remains in the upper chamber in the next period is at
most $D.$ Hence by induction we see that 
$$\mes_{\gamma}\{ x\in\gamma : \hbox{$\{P_U^m x\}_{m=0}^N$ 
visits good lines more than $k$ times} \}<D^k L$$

If instead the trajectory of $x$ visits good lines less than $k$ times in $N$ periods, then it has to visit consecutively $l$ bad lines at least once in $N$ periods.\\
Now suppose that the trajectory segment $x_n, \cdots, x_{n+l}$ land on bad lines $\gamma_n,\cdots,\gamma_{n+l}$ for some $n<N$ and we denote as $\mathcal{B}_n$ the set of all such $x\in\gamma$ that lands badly during $n$ to $n+l$ periods. We subdivide $\mathcal{B}_n$ into $\mathcal{B}_{n,L}$ and $\mathcal{B}_{n,S}$, where $\mathcal{B}_{n,L}$ collects points with $|\gamma_n| \ge \Lambda_u^{-l/2}$ and $\mathcal{B}_{n,S}$ collects points with $|\gamma_n| < \Lambda_u^{-l/2}$.\\
By Lemma \ref{lemma62}, $|\mathcal{B}_{n,S}| \le C^* \Lambda_u^{-l/2}$. On the other hand, it follows from uniform hyperbolicity that 
$$ \mes_{\gamma_n} \{\hbox{$x_n$ returns badly for next $l$ periods} \} = \frac{|\gamma_{n+l}|}{\Lambda_u^l} \le \frac{L}{\Lambda_u^l}.$$ 
Hence 
$$|\mathcal{B}_{n,L}| \le \frac{L}{\Lambda_u^l} \sum_{|\gamma_n| \ge \Lambda_u^{-l/2}} |G_U^{-n} \gamma_n| \le \frac{L}{\Lambda_u^l} |\gamma| \le \frac{L^2}{\Lambda_u^l}.$$

Combining the  estimates on $\mathcal{B}_{n,L}$ 
and $\mathcal{B}_{n,S}$, 
we have 
$$|\mathcal{B}_n| \le \frac{C^*}{\Lambda_u^{l/2}} + 
\frac{L^2}{\Lambda_u^l} \le \frac{C^*+L^2}{\Lambda_u^{l/2}}.$$
Consequently the set $\mathcal{B}$ of points on $\gamma$ which make $l$ consecutive bad landings is controlled in size by 
$$|\mathcal{B}| \le (kl+1) \frac{C^*+L^2}{\Lambda_u^{l/2}}.$$

Since a box $A$ is foliated by unstable lines, we conclude by a disintegration of measure argument and our choice of $k,l$ that under the linear map $G_U$ the set of points in $A$ which stay in the upper chamber for at least $N$ periods has measure less than $\epsilon.$ 
\end{proof}

\subsection{Quick Escape for the Actual Map.}
By Proposition \ref{prop43}, the fundamental domains of $P_{UU}$ are $\mathcal{O}(\mathcal{I}^{-1})$-deformation of the boxes $A,B$ and $P_{UU}=G_{UU}+\mathcal{O}(\mathcal{I}^{-1})$.\\
Now we prove Theorem \ref{thm2}.

\begin{proof}[Proof of Theorem \ref{thm2}]
Fix $\epsilon>0$ and a box $A$ with large energy $\tilde{\mathcal{I}}_0$ (to be specified later). By Proposition \ref{prop63} we choose $N$ such that in each $A$-box the points that remain in the upper chamber up to $N$ periods under the linear approximation $G_U$ take up a set of measure less than $0.5 \epsilon$, i.e. we take $N=kl+1$ where $k,l$ are integers such that 
\begin{equation}
\label{ChooseKL}
k>\frac{\log (0.25\epsilon/L)}{\log D}
\quad\text{and}\quad
\frac{kl+1}{\Lambda_u^{l/2}} < \frac{0.25\epsilon}{C^*+L^2}. 
\end{equation}
We shall show that the statement of Theorem \ref{thm2} holds with some large $\tilde{\mathcal{I}}_0=\tilde{\mathcal{I}}_0(\epsilon)$ and
\begin{equation}
\label{T-N}
T=2N.
\end{equation}

Let $A_n^\delta$ and $B_n^\delta$ denote the points in $A_n$ and $B_n$ which are closer than
$\delta$ to the boundary, 
$$\mathcal{A}^\delta=\bigcup_n A_n^\delta, \quad
\widetilde{\mathcal{A}}^\delta=\bigcup_n \left(A_n-\setminus A_n^\delta\right), $$
$$\mathcal{B}^\delta=\bigcup_n B_n^\delta, \quad
\widetilde{\mathcal{B}}^\delta=\bigcup_n \left(B_n-\setminus B_n^\delta\right).$$
 Choose $\delta<0.5 \epsilon$ so that the set of points in the box $A$ which visit either $\mathcal{A}^\delta$ or 
 $\mathcal{B}^\delta$ during the first $N$ iterations is less than $0.5 \epsilon.$
 
By Proposition \ref{prop43}, 
there is a constant $C_1$ such that if 
$P_{UU}^{12}(x)\in \widetilde{\mathcal{B}}^{C_1 /\mathcal{I}}$ and 
$P_{UU}^{21}(P_{UU}^{12} x)\in \widetilde{\mathcal{A}}^{C_1 /\mathcal{I}}$ 
then the orbit of $x$ stays in the upper chamber for the next period and 
$$ \left| P_U(x)-G_U(x)\right|\le 
\frac{C_1}{\mathcal{I}} \quad \text{where} \quad
P_U=P_{UU}^{21} \circ P_{UU}^{12}. $$

Accordingly there is a constant $C_2$ such that if for some $n\leq N$
\begin{equation}
\label{WellInside}
P_{UU}^{12} P_U^k (x)\in \widetilde{\mathcal{B}}^{\frac{C_2\Lambda_u^N}{ \mathcal{I}^*}}, \quad
P_U^{k+1} (x)\in \widetilde{\mathcal{A}}^{\frac{C_2 \Lambda_u^N}{ \mathcal{I}^*}}
\end{equation}
and $\mathcal{I}_k\ge \mathcal{I}^*$
for $k < n$ then the real orbit of $x$ stays in upper chamber for the first $n$ iterations
and
\begin{equation}
\label{Nearby}
 \left| P_U^n(x)-G_U^n (x)\right|\le \frac{C_2 \Lambda_u^N}{ \mathcal{I}^*}. 
 \end{equation}
 
Next, set $C_3=\frac{l_2^+ l_1^+}{l_2^- l_1^-} < 1.$ Then during $N$ iterations 
the value of $I$ cannot drop by more than $C_3^N$ times. Hence if $x$ satisfies \eqref{WellInside}
and $\mathcal{I}_0\geq C_3^N \mathcal{I}^*$ then \eqref{Nearby} holds.\\

Now choose $\mathcal{I}^*$ so that 
\begin{equation}
\label{IVeryLarge}
\frac{C_2 \Lambda_u^N}{\mathcal{I^*}}<\delta 
\end{equation}

Now we consider the orbits where $\bar{\mathcal{I}}\leq\mathcal{I}_0\leq \bar{\mathcal{I}}+1$
for some $\bar{\mathcal{I}}\geq C_3^N \mathcal{I}^*.$
There are three possibilities:

(i) The real orbit of $x$ leaves the upper chamber at some period $n<N$;

(ii) The real orbit of $x$ stays in $\widetilde{\mathcal{A}}^{\delta}$ for the first $N$ iterations;

(iii) The real orbit of $x$ stays in $\widetilde{\mathcal{A}}^{\delta}$ until it hits 
$\mathcal{A}^{\delta}\cup 
(P_{U}^{12})^{-1} \mathcal{B}^{\delta}$
at some period $n<N$.

Proposition \ref{prop63} and our choice of $\delta$ and $ \mathcal{I}^*$ imply that the set of orbits where either (ii) or (iii) happens has measure smaller than $\epsilon.$

This completes the proof of Theorem \ref{thm2}.
\end{proof}

\section{Conclusion}
\label{ScFinal}
We have described in this paper a two-dimensional exponential Fermi accelerator: a rectangular billiard with two moving slits. We found a mechanism for a particle to gain energy exponentially fast, i.e. the trapping regions. When the relative positions of two slits change at two critical jumps, a trapping region, either the upper or lower chamber, is created so that every high velocity orbit starts to gain energy exponentially fast once it gets trapped. We demonstrated that a trapping region exists for sizable choices of parameters and the exponential acceleration happens for almost all high energy orbits. Moreover under additional hyperbolicity assumptions on the parameters we provided an explicit estimate on the waiting time until which the exponential acceleration starts for most high-energy orbits.\\

It is worth noting that all the analysis done in this paper is based on the normal forms in the high energy region.
The normal form implies exponential energy growth almost surely if a particle starts with sufficiently high initial velocity, 
and it eliminates the possibility of oscillatory orbits. The normal forms do not apply in low energy 
region where we might have bounded orbits for certain wall motion. 
In this paper we did not analyze the case when a trapping region does not exist, which can be easily achieved by choosing parameters such that the relative positions of two slits do not change at two critical jumps. 
Our normal forms still apply even in this complicated case but the analysis would be more delicate as the particle needs to make a choice of traveling up or down every time it jumps. 
We also note that in the non-resonant case when the periods of the particle and the wall
are incommensurable, the normal form still applies, however the jumping time depends on the period.
In particular, the jumping times become dense on the period which precludes the existence of the trapping
region, so the problem becomes similar to the resonant non-trapping case.
These observations provide possible directions for future work.


\bibliography{billiardslit} 
\bibliographystyle{plain}

\end{document}